\documentclass[11pt,reqno]{amsart}
\usepackage[colorlinks=true,allcolors=blue,backref=page]{hyperref}
\usepackage{color}
\usepackage{amsmath, amssymb, amsthm}
\usepackage{mathrsfs}
\usepackage{mathtools}
\usepackage[noabbrev,capitalize,nameinlink]{cleveref}
\crefname{equation}{}{}
\usepackage{fullpage}
\usepackage[noadjust]{cite}
\usepackage{graphics}
\usepackage{pifont}
\usepackage{tikz}
\usepackage{bbm}
\usepackage[T1]{fontenc}

\usetikzlibrary{arrows.meta}

\usepackage{environ}
\usepackage{framed}
\usepackage{url}
\usepackage[linesnumbered,ruled,vlined]{algorithm2e}
\usepackage[noend]{algpseudocode}
\usepackage[labelfont=bf]{caption}
\usepackage{cite}
\usepackage{framed}
\usepackage[framemethod=tikz]{mdframed}
\usepackage{appendix}
\usepackage{graphicx}
\usepackage[textsize=tiny]{todonotes}
\usepackage{tcolorbox}
\usepackage{enumerate}
\usepackage[shortlabels]{enumitem}
\allowdisplaybreaks[1]

\apptocmd{\sloppy}{\hbadness 10000\relax}{}{} 

\crefname{algocf}{Algorithm}{Algorithms}

\crefname{equation}{}{} 
\crefname{conjecture}{Conjecture}{Conjectures} 
\AtBeginEnvironment{appendices}{\crefalias{section}{appendix}} 

\crefformat{enumi}{#2#1#3}
\crefrangeformat{enumi}{#3#1#4 to~#5#2#6}
\crefmultiformat{enumi}{#2#1#3}%
{ and~#2#1#3}{, #2#1#3}{ and~#2#1#3}

\usepackage[color,final]{showkeys} 

\colorlet{refkey}{orange!20}
\colorlet{labelkey}{blue!30}

\crefname{algocf}{Algorithm}{Algorithms}

\numberwithin{equation}{section}
\newtheorem{theorem}{Theorem}[section]
\newtheorem{proposition}[theorem]{Proposition}
\newtheorem{lemma}[theorem]{Lemma}
\newtheorem{claim}[theorem]{Claim}
\crefname{claim}{Claim}{Claims}

\newtheorem{conjecture}[theorem]{Conjecture}
\newtheorem*{question*}{Question}

\theoremstyle{definition}
\newtheorem{definition}[theorem]{Definition}

\newtheorem*{definition*}{Definition}

\theoremstyle{remark}
\newtheorem*{remark}{Remark}

\newcommand{\mb}{\mathbb}
\newcommand{\mbf}{\mathbf}
\newcommand{\mbm}{\mathbbm}
\newcommand{\mc}{\mathcal}

\newcommand{\mr}{\mathrm}

\newcommand{\ol}{\overline}
\newcommand{\on}{\operatorname}
\newcommand{\wh}{\widehat}

\allowdisplaybreaks

\title{Majority Dynamics: The Power of One}

\author[Sah]{Ashwin Sah}
\author[Sawhney]{Mehtaab Sawhney}

\address{Department of Mathematics, Massachusetts Institute of Technology, Cambridge, MA 02139, USA}
\email{\{asah,msawhney\}@mit.edu}

\thanks{Sah and Sawhney were supported by NSF Graduate Research Fellowship Program DGE-1745302.}

\begin{document}

\begin{abstract}
Consider $n=\ell+m$ individuals, where $\ell\le m$, with $\ell$ individuals holding an opinion $A$ and $m$ holding an opinion $B$. Suppose that the individuals communicate via an undirected network $G$, and in each time step, each individual updates her opinion according to a majority rule (that is, according to the opinion of the majority of the individuals she can communicate with in the network). This simple and well studied process is known as ``majority dynamics in social networks''. Here we consider the case where $G$ is a random network, sampled from the binomial model $\mathbb{G}(n,p)$, where $(\log n)^{-1/16}\le p\le 1-(\log n)^{-1/16}$. We show that for $n=\ell+m$ with $\Delta=m-\ell\le(\log n)^{1/4}$, the above process terminates whp after three steps when a consensus is reached. Furthermore, we calculate the (asymptotically) correct probability for opinion $B$ to ``win'' and show it is \[\Phi\bigg(\frac{p\Delta\sqrt{2}}{\sqrt{\pi p(1-p)}}\bigg) + O(n^{-c}),\]
where $\Phi$ is the Gaussian CDF. This answers two conjectures of Tran and Vu and also a question raised by Berkowitz and Devlin. 

The proof technique involves iterated degree revelation and analysis of the resulting degree-constrained random graph models via graph enumeration techniques of McKay and Wormald as well as Canfield, Greenhill, and McKay.
\end{abstract}

\maketitle

\section{Introduction}\label{sec:introduction}
Considerable effort has been devoted to understanding exchange of opinions between individuals, seeing as it plays a major role in all types of social interaction. Of course, no simple model can accurately describe the behavior of many actors in complicated situations, so analysis and understanding of natural models for this problem has generated significant interest. A natural model, which has even been of interest in biophysics \cite{MP43} and psychology \cite{CH56}, is so-called majority dynamics. It can be briefly described as follows. Given $n$ individuals, let the network $G$ capture the set of interactions between participants. For each participant $i\in\{1,\ldots,n\}$ with initial opinion $A_i^{(0)}\in\{\pm1\}$, at every time step they adopt the majority opinion of their neighbors, that is, $A_i^{(t+1)} = \on{sign}(\sum_{j\sim i}A_j^{(t)})$. The key object of study therefore is understanding the propagation of opinions and how the local structure of the network affects these dynamics. We refer the reader to \cite{MT17, MNT14, TV20, BD20} for further references regarding majority dynamics.

We now precisely define majority dynamics in terms of partitions of the graph $G$ as this will be our focus in order to analyze it. Additionally, following \cite{TV20} we follow the convention that if a participant's neighborhood is equally split between opinions then they keep the same opinion.
\begin{definition}\label{def:majority-dynamics}
Given a graph $G$ with bipartition $B_0\sqcup R_0$, the \emph{majority dynamics} at time $i$ are computed as follows. Given $B_i\sqcup R_i$, a new partition $B_{i+1}\sqcup R_{i+1}$ by swapping precisely those vertices with strictly more of their neighbors on the other side of the partition. We say a color blue or red \emph{wins by step} $k$ if $B_k = B_0\cup R_0$ or $R_k = B_0\cup R_0$, respectively.
\end{definition}

Our primary object of study in the paper concerns majority dynamics on random graphs $\mb{G}(n,p)$. First considered by Benjamini, Chan, O’Donnell, Tamuz, and Tan \cite{BCOTT16}, research has primarily focused on establishing that majority dynamics terminates in a finite number of steps (see e.g.~\cite{FKM20} and the very recent \cite{CKLT21} aimed at understanding sparse graphs) or understanding the even finer question of the distribution of which color majority dynamics terminates on \cite{TV20,BD20}

Our primary aim is to resolve a conjecture of Tran and Vu \cite{TV20} which informally states that for majority dynamics in $\mb{G}(n,1/2)$ even a bias of a single extra voter is sufficient to influence the final state by a positive probability. An essentially equivalent conjecture appears in subsequent work of Berkowitz and Devlin \cite[Conjecture~7]{BD20}.
\begin{conjecture}[{\cite[Conjecture~7]{TV20}}]\label{conj:main}
Majority dynamics on $\mb{G}(2n+\Delta,1/2)$ with sets $R_0 = \{v_1,\ldots,v_{n+\Delta}\}$ and $B_0 = \{v_1',\ldots,v_n'\}$ converges to $R_k = R_0\cup B_0$ with probability at least $1/2+f(\Delta)$, where $f(\Delta) > 0$, as $n\to\infty$.
\end{conjecture}

Tran and Vu \cite{TV20} resolved this conjecture for (even) $\Delta\ge 12$, and Berkowitz and Devlin \cite{BD20} resolved it for $\Delta\ge 3$.

We resolve this conjecture in full.
\begin{theorem}\label{thm:main}
There is an absolute constant $c > 0$ so that the following holds. Let $n\ge 1$. Let $0\le\Delta\le(\log n)^{1/4}$ and let $n\ge 1$, $(\log n)^{-1/16}\le p\le 1-(\log n)^{-1/16}$. In majority dynamics on $\mb{G}(2n+\Delta,p)$ with $|R_0| = n+\Delta$, with probability at least $1-O(n^{-c})$ there is a color with more vertices at step $1$ and that color wins by step $3$. Furthermore, $|R_3| = 2n+\Delta$ with probability
\[\Phi\bigg(\frac{p\Delta\sqrt{2}}{\sqrt{\pi p(1-p)}}\bigg) + O(n^{-c})\]
where $\Phi$ is the cdf of $\mc{N}(0,1)$.
\end{theorem}
\begin{remark}
In particular, the event that both colors have the same size at step $1$ occurs with decaying probability. The parameters $(\log n)^{-1/16}, (\log n)^{-1/4}$ can certainly be improved substantially but we have chosen to focus on the dense regime.
\end{remark}

We note that \cref{thm:main} additionally resolves \cite[Conjecture~8]{TV20} regarding monotonicity of the limiting probabilities with respect to $\Delta$, and the proof of \cref{thm:main} essentially answers \cite[Question~2]{BD20} (see in particular \cref{thm:day-one,thm:day-two} which provide fine information about the sizes of the various parts after one and two days). We also note that this is the first work which gives an exact limiting probability for a specific color winning when that probability is strictly between $0$ and $1$ (other than the simple symmetric case $\Delta = 0$).

We anticipate that the techniques of this paper combined with recent refined asymptotic enumeration results of \cite{LW17,LW20} can yield further refinements of work of \cite{BD20,FKM20}. In particular this may allow precise understanding of the number of steps before reaching stability for wider ranges of sparse $p$ than currently known.

\subsection{Strategy}\label{sub:strategy}
The broad structure of this paper breaks into $2$ phases. In the first we substantially refine results of \cite{BD20} in order to obtain a local limit theorem of how many vertices switched from blue to red and red to blue jointly. Our techniques rely extensively on graph enumeration results and models developed for degree sequences in $\mb{G}(n,p)$ by McKay and Wormald \cite{MW97} and random bipartite graphs by McKay and Skerman \cite{MS16}. Technically, this fine-tuned local limit theorem is not necessary to complete the analysis, and one can use a (non-joint) central limit theorem for the lead \cite[Theorem~1]{BD20} along with a precise computation of its mean (using for example techniques similar to \cref{app:calculations} or \cite[Lemma~12]{BD20}).

The second and third days also use the graph enumeration techniques of McKay and Wormald \cite{MW90} which were extended to bipartite graphs by Canfield, Greenhill, and McKay \cite{CGM08}; however at these stages we will only derive coarser information about the degree sequences and the number of red and blue vertices. In particular, we prove that given a sufficiently large initial lead, on the second day the number of red and blue vertices concentrate in intervals of length $O(n^{1-\eta})$ for an absolute constant $\eta$. Further, we show that the side leading will have developed a substantial lead (of linear order). Then a final application of degree enumeration implies that with high probability that the process terminates on the third day, because it is unlikely for any vertex to have a degree so large that it overcomes the gap between sizes. (Simpler arguments in \cite{BCOTT16,TV20,BD20} show termination by the fourth day without enumeration at this stage.)

For these two stages we rely on a modification of a concentration argument developed by the Ferber, Kwan, Narayanan, and the authors \cite{FKNSS21} where a general framework for applying the second moment method with McKay-Wormald \cite{MW90} enumeration formulas were used to resolve a conjecture of F\"uredi on the existence of ``unfriendly'' partitions in $G(n,1/2)$. The analysis here is substantially simpler as we need to track fewer parameters to guarantee convergence to termination within a finite time horizon. In particular, the analysis of the third day only requires a large-deviation bound on the degrees of vertices from what is expected in a degree-constrained random graph model, and the analysis of the second day has substantially simpler formulas due to the setting.

\subsection*{Acknowledgements}
We thank Asaf Ferber, Vishesh Jain, Matthew Kwan, and Bhargav Narayanan for discussions related to this project.

\section{Day one}\label{sec:one}
As mentioned, the analysis for the first day involves proving a local limit theorem for the sizes of parts. Although a central limit theorem was shown by Berkowitz and Devlin \cite{BD20} for the size of the red partition after one step, we will require understanding of how many vertices switched from blue to red and red to blue jointly, rather than the net amount. A central limit theorem may be derivable from their method, which involves moments. We need only a joint central limit theorem but we have chosen to demonstrate a local limit theorem to demonstrate the power of these techniques, and due to its independent interest. In particular, enumeration techniques allow one to reduce this computation to a local limit theorem for certain binomial random variables and various question about the model can be derived using these techniques.

The main result of this section is the following theorem. Its proof occupies \cref{sub:initial-estimates,sub:graph-models,sub:local-limit,sub:integration}.
\begin{theorem}\label{thm:day-one}
There exists constants $C,c>0$ such that the following holds. Let $n\ge 2$, let $0\le\Delta\le(\log n)^{1/4}$ and let $(\log n)^{-1/4}\le p\le 1-(\log n)^{-1/4}$. Let
\begin{align*}
\sqrt{\frac{n}{4\pi}}x' &= x - \bigg(\frac{1}{2}+\frac{p(\Delta-1)+1/2}{2\sqrt{\pi p(1-p)n}}\bigg)n,\\
\sqrt{\frac{n}{4\pi}}y' &= y - \bigg(\frac{1}{2}+\frac{p(-\Delta-1)+1/2}{2\sqrt{\pi p(1-p)n}}\bigg)n.
\end{align*}
In majority dynamics on $\mb{G}(2n+\Delta,p)$ with $|R_0| = n+\Delta$, we have that 
\begin{align*}
\mb{P}[|R_0\cap R_1| = x&\wedge|B_0 \cap B_1|=y]\\
&= \frac{2}{n\sqrt{\pi(2+\pi)}}\exp\bigg(-\frac{(1+\pi)(x')^2-2(x'y')+(1+\pi)(y')^2}{2\pi(2+\pi)}\bigg) + O(n^{-1-c}).
\end{align*}
Furthermore for $|x'|$ or $|y'|\ge C\sqrt{\log n}$ we have that
\[\mb{P}[|R_0\cap R_1| = x\wedge|B_0 \cap B_1|=y]\le n^{-5}.\]
\end{theorem}

\subsection{Initial estimates}\label{sub:initial-estimates}
We will first need some initial estimates regarding specific distributions which will show up when computing our local limit theorem. First, we record the probability that one binomial is greater than a different binomial with similar parameters. We defer its proof, which consists mainly of binomial manipulations and applications of well-known local central limit theorems, to \cref{app:calculations}.
\begin{lemma}\label{lem:chop-probability}
There is $c > 0$ so that the following holds. We are given $n\ge 2$, $\tau\in\mb{Z}$ of magnitude at most $2(\log n)^{1/4}$, and $(\log n)^{-1/4}\le p\le 1-(\log n)^{-1/4}$. Suppose that $q = p + \alpha/n$ and $q' = p + \beta/n$ with $|\alpha|,|\beta|\le 40\sqrt{p(1-p)\log n}$. Then
\[\mb{P}[\on{Bin}(n+\tau,q)\ge\on{Bin}(n,q')] = \frac{1}{2}+\frac{p\tau + 1/2 + \alpha - \beta}{2\sqrt{\pi p(1-p)n}} + O(n^{-3/4}).\]
\end{lemma}

Next, we need to understand the mean and standard deviation of certain conditioned binomial random variables. The level of control required can be deduced from the Berry-Esseen theorem.
\begin{lemma}\label{lem:chop-statistics}
There is $c > 0$ so that the following holds. We are given $n\ge 2$, $\tau\in\mb{Z}$ of magnitude at most $2(\log n)^{1/4}$, and $(\log n)^{-1/4}\le p\le 1-(\log n)^{-1/4}$. Suppose that $q,q'\in p\pm 40\sqrt{p(1-p)\log n}/n$. Let $X\sim\on{Bin}(n+\tau,q)$ and $Y\sim\on{Bin}(n,q')$. Let $X^+$ be $X$ conditional on $X > Y$ while $X^-$ be $X$ conditional on $X\le Y$. Then
\begin{align*}
\mb{E}X^+ = pn + \sqrt{\frac{p(1-p)n}{\pi}} + O(n^{1/4}),&\qquad\on{Var}X^+ = \bigg(1-\frac{1}{\pi}\bigg)p(1-p)n+O(n^{3/4})\\
\mb{E}X^- = pn - \sqrt{\frac{p(1-p)n}{\pi}} + O(n^{1/4}),&\qquad\on{Var}X^- = \bigg(1-\frac{1}{\pi}\bigg)p(1-p)n+O(n^{3/4}).
\end{align*}
\end{lemma}
\begin{proof}
By Berry-Esseen, the joint distribution $(X-pn,Y-pn)/\sqrt{p(1-p)n}$ has cumulative distribution function differing from $\mc{N}(0,I_2)$ by $O(1/\sqrt{np(1-p)})$ pointwise. (Note that $\tau$ is small, so the shifts are negligible.) Let $Z_1,Z_2\sim\mc{N}(0,1)$. We see that
\begin{align*}
\mb{E}X^+ = \mb{E}[X|X\ge Y] &= pn + \sqrt{p(1-p)n}\frac{\mb{E}[X\mbm{1}_{X\ge Y}]}{\mb{P}[X\ge Y]}\\
&= pn + \sqrt{p(1-p)n}\frac{\mb{E}[Z_1\mbm{1}_{Z_1\ge Z_2}] + O(n^{-1/4})}{\frac{1}{2} + O(n^{-1/4})}\\
&= pn + \sqrt{\frac{p(1-p)n}{\pi}} + O(n^{1/4}).
\end{align*}
The error terms $O(n^{-1/4})$ come from integrating the discrepancy in cumulative distribution functions over the region where $(X-pn,Y-pn)/\sqrt{p(1-p)n}$ is bounded by $O(\sqrt{\log n})$ and using a large deviation bound for binomials outside. Similarly,
\begin{align*}
\mb{E}(X^+-pn)^2 = \mb{E}[(X-pn)^2|X\ge Y] &= p(1-p)n\frac{\mb{E}[Z_1^2\mbm{1}_{Z_1\ge Z_2}] + O(n^{-1/4})}{\frac{1}{2} + O(n^{-1/4})}\\
&= p(1-p)n + O(n^{3/4}).
\end{align*}
Therefore
\[\on{Var}X^+ = \mb{E}[(X^+-pn)^2] - (\mb{E}[X^+-pn])^2 = \bigg(1-\frac{1}{\pi}\bigg)p(1-p)n + O(n^{3/4}).\]
Above, we used $\mb{E}[Z_1\mbm{1}_{Z_1\ge Z_2}] = 1/(2\sqrt{\pi})$ and $\mb{E}[Z_1^2\mbm{1}_{Z_1\ge Z_2}] = 1/2$. The computation for $X^-$ is exactly analogous so we omit it.
\end{proof}

Next, we need a local limit theorem for sums of these conditioned binomial random variables. The proof uses log-concavity of binomial distributions, along with a technique of Bender \cite{Ben73} which upgrades a Berry-Essen quality central limit theorem for a log-concave variable into a local central limit theorem. Though it follows by directly citing such results, we spell out the details in order to quantify the bounds.
\begin{proposition}\label{prop:chopped-lclt}
There is $c > 0$ so that the following holds. We are given $n\ge 2$, $m,\tau_k,\tau_k'\in\mb{Z}$ of magnitude at most $2(\log n)^{1/4}$, and $(\log n)^{-1/4}\le p\le 1-(\log n)^{-1/4}$. Suppose that $q_k,q_k'\in p\pm 40\sqrt{p(1-p)\log n}/n$. Let $X_k\sim\on{Bin}(n+\tau_k,q_k)$ and $Y_k\sim\on{Bin}(n+\tau_k',q_k')$. Let $X_k^+$ be $X_k$ conditional on $X_k > Y_k$ while $X_k^-$ be $X_k$ conditional on $X_k\le Y_k$. Fix some $i\in[n+m]$ and sequence $\epsilon_k\in\{\pm1\}$, and let
\[S = \sum_{k=1}^i\epsilon_kX_k^- + \sum_{k=i+1}^{n+\tau'}\epsilon_kX_k^+.\]
Then
\[\mb{P}[S = s] = \frac{1}{\sqrt{2\pi}\sigma_S}\exp\bigg(-\frac{(s-\mu_S)^2}{2\sigma_S^2}\bigg) + O\bigg(\frac{1}{n^{1/5}\sigma_S}\bigg)\]
for all $i\in[n+m]$ and $s\in\mb{Z}$, if $\mu_S$ and $\sigma_S$ are the mean and variance of $S$.
\end{proposition}
\begin{proof}
Note that $X_k,Y_k$ have probabilities converging to that of a normalized Gaussian, by a local limit theorem. Combining with tail bounds, we easily see that $X_k^+$ has well-behaved (centered) moments: its variance is $\Theta(p(1-p)n)$ and its centered third moment is $\Theta((p(1-p)n)^{3/2})$. The same holds for $X_k^-$. Therefore, the Berry-Esseen theorem shows that the cumulative distribution functions of $S$ and $\mc{N}(\mu_S,\sigma_S^2)$ differ by $O(1/\sqrt{n})$ everywhere.

Next, note that $X_k,Y_k$ have log-concave probability mass functions (on $\mb{Z}$) by log-concavity of binomials, hence $(X_k,Y_k)$ has a jointly log-concave probability mass function in the sense that
\[p(a,b)p(c,d)\le p\bigg(\bigg\lfloor\frac{a+c}{2}\bigg\rfloor,\bigg\lfloor\frac{b+d}{2}\bigg\rfloor\bigg)p\bigg(\bigg\lceil\frac{a+c}{2}\bigg\rceil,\bigg\lceil\frac{b+d}{2}\bigg\rceil\bigg).\]
Conditioning on a convex set preserves log-concavity in this sense, hence $(X_k,Y_k)$ conditional on $X_k\ge Y_k$ as well as conditional on $X_k < Y_k$ both have log-concave probability mass functions. By \cite[Theorem~1.2]{HKS19} (which is essentially reproves to \cite[Theorem~1.4]{KL19} but allows functions to be $0$), we see that the marginals of a distribution which is log-concave in this sense are log-concave. Therefore $X_k^+$, $X_k^-$ have log-concave probability mass functions.

Finally, convolutions of log-concave sequences are log-concave, so $S$ has log-concave probability mass function. We established earlier that it satisfies a quantitative central limit theorem. We now quantify an argument of Bender \cite{Ben73} in order to deduce the desired result.

Let $m_{S}$ be the mode of $S$. Above this value, the probability mass is nonincreasing, while below it is nondecreasing. First suppose that $s > m_{S} + n^{-1/4}\sigma_S$. We see that
\begin{align*}
\mb{P}[S = s]&\le\frac{1}{\lceil n^{-1/4}\sigma_S\rceil}\mb{P}[s\le S < s+n^{-1/4}\sigma_S]\\
&= \frac{1}{\lceil n^{-1/4}\sigma_S\rceil}\mb{P}[s\le\mc{N}(\mu_S,\sigma_S^2) < s+n^{-1/4}\sigma_S] + O(n^{-1/4}\sigma_S^{-1})\\
&= \frac{1}{\sqrt{2\pi}\sigma_S}\exp\bigg(-\frac{(s-\mu_S)^2}{2\sigma_S^2}\bigg) + O(n^{-1/5}\sigma_S^{-1}).
\end{align*}
The last line follows since $(s-\mu_S)^2/(2\sigma_S^2)$ is either stable up to a multiplicative factor of $(1+O(n^{-1/5}))$ upon changing $s$ by $\pm n^{-1/4}\sigma_S$ or is super-polynomially small (hence absorbed into the additive error term, since $\sigma_S^2 = \Theta(p(1-p)n^2)$ is polynomial). The lower bound is analogous. Furthermore, this holds for $s < m_{S} - n^{-1/4}\sigma_S$ by an identical argument. Therefore,
\[\mb{P}[S = s] = \frac{1}{\sqrt{2\pi}\sigma_S}\exp\bigg(-\frac{(s-\mu_S)^2}{2\sigma_S^2}\bigg) + O(n^{-1/5}\sigma_S^{-1})\]
as long as $s\notin m_{S}\pm n^{-1/4}\sigma_S$.

Finally, suppose that $m_{S}\le s\le m_{S}+n^{-1/4}\sigma_S$ (the symmetric case is analogous). We have
\begin{align*}
\mb{P}[S = s]&\ge\mb{P}[S = s + \lceil n^{-1/4}\sigma_S\rceil]\\
&= \frac{1}{\sqrt{2\pi}\sigma_S}\exp\bigg(-\frac{(s+\lceil n^{-1/4}\sigma_S\rceil-\mu_S)^2}{2\sigma_S^2}\bigg) + O(n^{-1/5}\sigma_S^{-1})\\
&= \frac{1}{\sqrt{2\pi}\sigma_S}\exp\bigg(-\frac{(s-\mu_S)^2}{2\sigma_S^2}\bigg) + O(n^{-1/5}\sigma_S^{-1}),
\end{align*}
where the last equality uses a similar argument to above. This is in fact enough to demonstrate that $|m_{S}-\mu_S| = O(n^{-1/5}\sigma_S)$ (since if it were too far, then the sequence would have an increase-decrease pattern twice).

Finally, we obtain an upper bound via log-concavity:
\begin{align*}
\mb{P}[S = s]\le\frac{\mb{P}[S = s + \lceil n^{-1/4}\sigma_S\rceil]^2}{\mb{P}[S = s + 2\lceil n^{-1/4}\sigma_S\rceil]} = \frac{1}{\sqrt{2\pi}\sigma_S}\exp\bigg(-\frac{(s-\mu_S)^2}{2\sigma_S^2}\bigg) + O(n^{-1/5}\sigma_S^{-1})
\end{align*}
by an analogous computation and the fact that $m_{S},\mu_S$ are close. The result follows.
\end{proof}

\subsection{Degree sequence models}\label{sub:graph-models}
We now define a plethora of degree sequence models for random graphs that will be needed for the computations. At a high level, the work of McKay and Wormald \cite{MW97} and McKay and Skerman \cite{MS16} demonstrate that degrees of random graphs look independent conditional on, for example, total edge count. These models provide a way to encapsulate these facts quantitatively.
\begin{definition}[Degree sequence domains]\label{def:degree-sets}
Let $I_n = \{0,\ldots,n-1\}^n$, $E_n$ be the even sum sequences in this set, and $I_n^\ell$ be the sum $\ell$ sequences. We will typically denote elements of these sets by $\mbf{d}$. Let $I_{m,n} = \{0,\ldots,n\}^m\times\{0,\ldots,m\}^n$, $E_{m,n}$ be the sequences with equal sums on both sides, and $E_{m,n}^\ell$ be the sequences with equal sums $\ell$. We will typically denote elements of these sets by $\mbf{s}$ of length $m$ and $\mbf{t}$ of length $n$. We will denote random variable versions of these by capital boldface instead.
\end{definition}
\begin{definition}[True degree models]\label{def:true-model}
$\mc{D}_p^n$ is the degree sequence distribution of $\mb{G}(n,p)$, which is a random variable supported on $E_n\subseteq I_n$. $\mc{D}_p^{m,n}$ is the degree sequence distribution of a bipartite graph with $m$ vertices on one side and $n$ on the other, each edge included independently with probability $p$, which is a random variable supported on $E_{m,n}\subseteq I_{m,n}$.
\end{definition}
\begin{definition}[Independent degree models]\label{def:independent-model}
$\mc{B}_p^n$ is the distribution of $n$ independent $\on{Bin}(n-1,p)$ random variables, supported on $I_n$. $\mc{B}_p^{m,n}$ is the distribution of $m$ independent $\on{Bin}(n,p)$ and $n$ independent $\on{Bin}(m,p)$ variables, supported on $I_{m,n}$.
\end{definition}
\begin{definition}[Conditioned degree models]\label{def:conditioned-model}
$\mc{E}_p^n$ is the distribution of $\mc{B}_p^n$ conditioned on having even sum, supported on $E_n$. $\mc{E}_p^{m,n}$ is the distribution of $\mc{B}_p^{m,n}$ conditioned on having equal sums on both sides, supported on $E_{m,n}$.
\end{definition}
\begin{definition}[Integrated degree models]\label{def:integrated-model}
$\mc{I}_p^n$ is the distribution sampled as follows. Sample $p'\sim\mc{N}(p,p(1-p)/(n^2-n))$, conditional on being in $(0,1)$. Then sample from $\mc{E}_{p'}^n$. $\mc{I}_p^{m,n}$ is the distribution sampled as follows. Sample $p'\sim\mc{N}(p,p(1-p)/(2mn))$, conditional on being in $(0,1)$. Then sample from $\mc{E}_{p'}^{m,n}$.
\end{definition}

We are now ready to state the necessary results.
\begin{theorem}[{From~\cite[Theorem~3(ii)]{MW90},~\cite[Theorem~3.6]{MW97}}]\label{thm:graph-indep}
There is $c > 0$ and a growing function so that the following holds. Let $n\ge 2$ and suppose $(\log n)^{-1/4}\le p\le 1-(\log n)^{-1/4}$. There is an event $B_p^n\subseteq I_n$ such that $\mb{P}_{\mc{D}_p^n}[B_p^n] = n^{-\omega(1)}$ and uniformly for all $\mbf{d}\in I_n\setminus B_p^n$ we have
\[\mb{P}_{\mc{D}_p^n}[\mbf{D}=\mbf{d}] = (1+O(n^{-c}))\mb{P}_{\mc{I}_p^n}[\mbf{D} = \mbf{d}]\]
\end{theorem}
\begin{theorem}[{From~\cite[Theorem~1(a)]{MS16}}]\label{thm:bipartite-indep}
There is $c > 0$ so that the following holds. Suppose $m,n\ge 2$ are such that $m = O(n\sqrt{\log n})$ and $n = O(m\sqrt{\log m})$. Suppose that $(\log n)^{-1/4}\le p\le 1-(\log n)^{-1/4}$. Then there is an event $B_p^{m,n}\subseteq I_{m,n}$ such that $\mb{P}_{\mc{D}_p^{m,n}}[B_p^{m,n}] = O(\exp(-n^c))$ and uniformly for $(\mbf{s},\mbf{t})\in I_{m,n}\setminus B$ we have
\[\mb{P}_{\mc{D}_p^{m,n}}[\mbf{S} = \mbf{s}\wedge\mbf{T} = \mbf{t}] = (1+O(n^{-1/3}))\mb{P}_{\mc{I}_p^{m,n}}[\mbf{S} = \mbf{s}\wedge\mbf{T} = \mbf{t}].\]
\end{theorem}

\subsection{Computing a local limit result}\label{sub:local-limit}
\subsubsection{Transferring to an independent model}\label{sub:transfer}
Now consider sampling $\mb{G}(2n+\Delta,p)$ and revealing the degrees among each part $R_0$ and $B_0$ as well from vertices in $R_0$ to $B_0$ and vice versa. We swap vertices purely based on this degree information. Since the sizes of the swapped parts are measurable with respect to this, which has distribution coming from three independent Erd\H{o}s-Renyi graph models, we see by \cref{thm:graph-indep,thm:bipartite-indep} that up to a multiplicative factor of $1+O(n^{-c})$ and an additive error of $n^{-\omega(1)}$ it is enough to compute the relevant probabilities if the models on the parts are $\mc{I}_p^{n+\Delta}$, $\mc{I}_p^n$, and $\mc{I}_p^{n+\Delta,n}$ instead. We let $\mbf{d}$ be the degree sequence of size $n+\Delta$, $\mbf{d}'$ be the one of length $n$, and $\mbf{s},\mbf{t}$ be of length $m=n+\Delta$ and $n$.

At this point it is useful to define $R_0 = \{n+1,\ldots,2n+\Delta\}$ and $B_0 = [n]$ as usual and define the swapped sets $R_1,B_1$ purely as functions of a triple of degree sequences $(\mbf{d},\mbf{d}',(\mbf{s},\mbf{t}))$ from $I_{n+\Delta}$, $I_n$, and $I_{n+\Delta,n}$. (We define it in the obvious way so as to apply even if the total sum in $I_n$ is not even, or the sums across both sides in $I_{n+\Delta,n}$ are not equal.)

With this in mind, the transference described above can be written quantitatively as
\begin{align}
\mb{P}_{\substack{\mc{D}_p^{n+\Delta},\mc{D}_p^n\\\mc{D}_p^{n+\Delta,n}}}[|R_0\cap R_1| = x\wedge|B_0\cap B_1| = y] = (1+O(n^{-c}))\mb{P}_{\substack{\mc{I}_p^{n+\Delta},\mc{I}_p^n\\\mc{I}_p^{n+\Delta,n}}}[&|R_0\cap R_1| = x\wedge|B_0\cap B_1| = y]\notag\\
&+ O(n^{-\omega(1)}).\label{eq:D-I}
\end{align}
Furthermore,
\begin{align}
&\mb{P}_{\substack{\mc{I}_p^{n+\Delta},\mc{I}_p^n\\\mc{I}_p^{n+\Delta,n}}}[|R_0\cap R_1| = x\wedge|B_0\cap B_1| = y]\notag\\
&= \frac{1}{\int_{q_0,q_1,q_2\in[0,1]}d\mu(q_0,q_1,q_2)}\int_{q_0,q_1,q_2\in[0,1]}\mb{P}_{\substack{\mc{E}_{q_0}^{n+\Delta},\mc{E}_{q_1}^n\\\mc{E}_{q_2}^{n+\Delta,n}}}[|R_0\cap R_1| = x\wedge|B_0\cap B_1| = y]d\mu(q_0,q_1,q_2)\notag\\
&= \int_{q_0,q_1,q_2\in p\pm20\sqrt{p(1-p)\log n}/n}\mb{P}_{\substack{\mc{E}_{q_0}^{n+\Delta},\mc{E}_{q_1}^n\\\mc{E}_{q_2}^{n+\Delta,n}}}[|R_0\cap R_1| = x\wedge|B_0\cap B_1| = y]d\mu(q_0,q_1,q_2) + O(n^{-10}),\label{eq:I-E}
\end{align}
where $\mu$ denotes the measure of three independent Gaussians centered at $p$ with variances $p(1-p)/((n+\Delta)^2-(n+\Delta))$, $p(1-p)/(n^2-n)$, and $p(1-p)/(2n(n+\Delta))$. The last line follows since such Gaussians lie in $(0,1)$ with exponentially good probability, and in fact are of size $p\pm20\sqrt{p(1-p)\log n}/n$ with probability at least $1-n^{-10}$.

At this point, we have nearly reached a model with independent Bernoulli sequences. However, we must condition on being even sum or having equal sum across two sides. To deal with this, we iteratively apply Bayes's rule to reduce to understanding genuinely independent random variables. This technique is closely related that in the proof given for \cite[Theorem~8]{MS16}. We have
\begin{align}
\mb{P}_{\substack{\mc{E}_{q_0}^{n+\Delta},\mc{E}_{q_1}^n\\\mc{E}_{q_2}^{n+\Delta,n}}}&[|R_0\cap R_1| = x\wedge|B_0\cap B_1| = y]\notag\\
&=\frac{\mb{P}_{\mc{B}_{q_0}^{n+\Delta},\mc{B}_{q_1}^n,\mc{B}_{q_2}^{n+\Delta,n}}[|R_0\cap R_1| = x\wedge|B_0\cap B_1| = y\wedge |\mbf{D}|/2,|\mbf{D}'|/2\in\mb{Z}\wedge|\mbf{S}| = |\mbf{T}|]}{\mb{P}_{\mc{B}_{q_0}^{n+\Delta},\mc{B}_{q_1}^n,\mc{B}_{q_2}^{n+\Delta,n}}[|\mbf{D}|/2,|\mbf{D}'|/2\in\mb{Z}\wedge|\mbf{S}| = |\mbf{T}|]}.\label{eq:E-B-even}
\end{align}
At this point, every event being considered is essentially coming from a sum of independent binomials or counting inequalities between independent binomials, so one should expect that these probabilities can be computed precisely. We can in fact do this, although we choose to iteratively simplify the expression by removing portions that ``act independent''.

\subsubsection{Removing evenness}\label{sub:evenness}
First, reveal $\mc{B}_{q_2}^{n+\Delta,n}$, that is, $\mbf{S}$ and $\mbf{T}$. Further reveal $R_0\cap R_1$ and $B_0\cap B_1$. Clearly the remaining randomness is as follows: for $v\in R_0\cap R_1$, we sample $d_v\sim\on{Bin}(n+\Delta,p_0)|_{\ge s_v}$, and similar for the other three parts. Note that with probability at least $1-2\exp(-\Omega(n))$ there are at least $n/4$ vertices $v\in R_0$ with $s_i\in pn\pm 100\sqrt{p(1-p)n}$ and at least $n/4$ vertices $v\in B_0$ with $t_i\in pn\pm 100\sqrt{p(1-p)n}$. For such vertices, regardless of whether it was revealed to be in $R_0\cap R_1$ or $R_0\setminus R_1$ (and similar for blue vertices), we see that the conditional distribution of its degree is some conditioned binomial that is easily checked to be equidistributed $(\mr{mod}~2)$ up to say an error of $O(n^{-1/4})$. If we reveal the degrees of every other vertex, then add up $n/4$ of these random variables, we obtain equidistribution $(\mr{mod}~2)$ where both values are attained with probability $1/2+O(\exp(-n))$. Therefore the numerator and denominator satisfy
\begin{align}
&\frac{\mb{P}_{\mc{B}_{q_0}^{n+\Delta},\mc{B}_{q_1}^n,\mc{B}_{q_2}^{n+\Delta,n}}[|R_0\cap R_1| = x\wedge|B_0\cap B_1| = y\wedge |\mbf{D}|/2,|\mbf{D}'|/2\in\mb{Z}\wedge|\mbf{S}| = |\mbf{T}|]}{\mb{P}_{\mc{B}_{q_0}^{n+\Delta},\mc{B}_{q_1}^n,\mc{B}_{q_2}^{n+\Delta,n}}[|\mbf{D}|/2,|\mbf{D}'|/2\in\mb{Z}\wedge|\mbf{S}| = |\mbf{T}|]}\notag\\
&= \frac{(\frac{1}{4}+O(\exp(-n)))\mb{P}_{\mc{B}_{q_0}^{n+\Delta},\mc{B}_{q_1}^n,\mc{B}_{q_2}^{n+\Delta,n}}[|R_0\cap R_1| = x\wedge|B_0\cap B_1| = y\wedge|\mbf{S}| = |\mbf{T}|] + O(\exp(-\Omega(n)))}{(\frac{1}{4}+O(\exp(-n)))\mb{P}_{\mc{B}_{q_2}^{n+\Delta,n}}[|\mbf{S}| = |\mbf{T}|] + O(\exp(-\Omega(n)))}\notag\\
&= \frac{\mb{P}_{\mc{B}_{q_0}^{n+\Delta},\mc{B}_{q_1}^n,\mc{B}_{q_2}^{n+\Delta,n}}[|R_0\cap R_1| = x\wedge|B_0\cap B_1| = y\wedge|\mbf{S}| = |\mbf{T}|]}{\mb{P}_{\mc{B}_{q_2}^{n+\Delta,n}}[|\mbf{S}| = |\mbf{T}|]} + O(\exp(-\Omega(n))).\label{eq:B-even-B}
\end{align}
In the last line, we used that the final denominator probability is large. This can be seen since it is the chance that two samples of $\on{Bin}(n(n+\Delta),q_2)$ equal each other. Being the same distribution supported on $[0,n(n+\Delta)]$, we see this occurs with probability at least $1/(n(n+\Delta)+1)$ by Cauchy-Schwarz.

\subsubsection{Computing the numerator}\label{sub:numerator}
In fact, this denominator can be computed precisely using a local limit theorem for binomial random variables. We therefore focus attention on computing the numerator. We have
\begin{align}
&\mb{P}_{\mc{B}_{q_0}^{n+\Delta},\mc{B}_{q_1}^n,\mc{B}_{q_2}^{n+\Delta,n}}[|R_0\cap R_1| = x\wedge|B_0\cap B_1| = y\wedge|\mbf{S}| = |\mbf{T}|]\notag\\
&= \sum_{|A|=x,|B|=y}\mb{P}_{\mc{B}}[R_0\cap R_1 = A\wedge B_0\cap B_1 = B]\mb{P}_{\mc{B}}[|\mbf{S}| = |\mbf{T}||R_0\cap R_1 = A\wedge B_0\cap B_1 = B].\label{eq:B-B-cond}
\end{align}
We can exactly compute the distribution of $|R_0\cap R_1|$ and $|B_0\cap B_1|$, which are independent. We make the following definitions for convenience going forward:
\begin{itemize}
    \item $q_i = p + \alpha_i/n$ for $0\le i\le 2$, where $|\alpha_i|\le 20\sqrt{p(1-p)\log n}$;
    \item $X_k\sim\on{Bin}(n+\Delta-1,q_0)$ and $Y_k\sim\on{Bin}(n,q_2)$ for $k\in[n+\Delta]$;
    \item $Y_k^-$ is the distribution of $Y_k$ conditional on $Y_k\le X_k$ and $Y_k^+$ is conditional on $Y_k > X_k$;
    \item $Z_k\sim\on{Bin}(n+\Delta,q_2)$ and $W_k\sim\on{Bin}(n-1,q_1)$ for $k\in[n]$;
    \item $Z_k^-$ is $Z_k$ conditioned on $Z_k\le W_k$ and $Z_k^+$ is conditioned on $Z_k > W_k$;
    \item $r = \mb{P}[X_1\ge Y_1]$ and $b = \mb{P}[W_1\ge Z_1]$.
\end{itemize}

We have
\begin{align*}
|R_0\cap R_1| &= \sum_{k=1}^{n+\Delta}\mbm{1}[X_k\ge Y_k]\sim\on{Bin}(n+\Delta,r),\\
|B_0\cap B_1| &= \sum_{k=1}^n\mbm{1}[W_k\ge Z_k]\sim\on{Bin}(n,b).
\end{align*}
Additionally, we can compute the distributions of $|\mbf{S}|$ and $|\mbf{T}|$ conditional on $A = R_0\cap R_1$ and $B = B_0\cap B_1$, which are independent. It actually only depends on the sizes. If we condition on $|A| = x$ and $|B| = y$, we have
\begin{align*}
|\mbf{S}|&\sim\sum_{k=1}^x Y_k^- + \sum_{k=x+1}^{n+\Delta}Y_k^+,\\
|\mbf{T}|&\sim\sum_{k=1}^y Z_k^- + \sum_{k=y+1}^n Z_k^+.
\end{align*}
At this point, computing the probability that $|\mbf{S}| - |\mbf{T}| = 0$ amounts to proving a local central limit theorem for all possible mixed sums and differences of these independent random variables. We have already done this in \cref{prop:chopped-lclt}. Explicitly, this means that for $|A| = x$ and $|B| = y$ that
\begin{align*}
\mb{P}_{\mc{B}}[|\mbf{S}| = |\mbf{T}||R_0\cap R_1 = A\wedge B_0\cap B_1 = B] &= \frac{1}{\sqrt{2\pi}\sigma_{x,y}}\exp\bigg(\frac{\mu_{x,y}^2}{2\sigma_{x,y}^2}\bigg) + O(n^{-1/5}\sigma_{x,y}^{-1})
\end{align*}
where $\mu_{x,y}$, $\sigma_{x,y}^2$ are the mean and variance of $|\mbf{S}|-|\mbf{T}|$ conditional on $|A| = x$ and $|B| = y$.

By \cref{lem:chop-statistics}, we have
\[\sigma_{x,y}^2 = \bigg(2-\frac{2}{\pi}\bigg)p(1-p)n^2 + O(n^{7/4}).\]
Therefore, let $\sigma = \sqrt{(2-2/\pi)p(1-p)}n$ and note that for $|A| = x$ and $|B| = y$ we have
\begin{align*}
\mb{P}_{\mc{B}}[|\mbf{S}| = |\mbf{T}||R_0\cap R_1 = A\wedge B_0\cap B_1 = B] &= \frac{1}{\sqrt{2\pi}\sigma}\exp\bigg(\frac{\mu_{x,y}^2}{2\sigma^2}\bigg) + O(n^{-1/5}\sigma^{-1}).
\end{align*}

It remains to understand $\mu_{x,y}$. We have
\begin{align*}
\mu_{x,y} = x\mb{E}Y_k^- + (n+\Delta-x)\mb{E}Y_k^+ - y\mb{E}Z_k^- - (n-y)\mb{E}Z_k^+.
\end{align*}

\begin{claim}\label{clm:mu-x-y}
If $|x-n/2|,|y-n/2|\le\sqrt{n}\log n$ we have
\[\mu_{x,y} = \bigg(\frac{\alpha_1-\alpha_0-2p\Delta}{\pi}\bigg)n + 2\sqrt{\frac{p(1-p)n}{\pi}}(x-y) + O(n^{4/5}).\]
\end{claim}
\begin{proof}
From \cref{lem:chop-statistics} we have
\begin{align*}
\mb{E}Y_k^+ &= pn + \sqrt{\frac{p(1-p)n}{\pi}} + O(n^{1/4}) = \mb{E}Z_k^+,\\
\mb{E}Y_k^- &= pn - \sqrt{\frac{p(1-p)n}{\pi}} + O(n^{1/4}) = \mb{E}Z_k^-.
\end{align*}
The error terms are not good enough to do a direct replacement. However, from \cref{lem:chop-probability} we have
\begin{align*}
r &= \mb{P}[X_1\ge Y_1] = \frac{1}{2} + \frac{p(\Delta-1) + 1/2 + \alpha_0 - \alpha_2}{2\sqrt{\pi p(1-p)n}} + O(n^{-3/4}),\\
b &= \mb{P}[W_1\ge Z_1] = \frac{1}{2} + \frac{p(-\Delta-1) + 1/2 + \alpha_1 - \alpha_2}{2\sqrt{\pi p(1-p)n}} + O(n^{-3/4})
\end{align*}
and additionally, by definition,
\begin{align*}
r\mb{E}Y_k^- + (1-r)\mb{E}Y_k^+ &= \mb{E}Y_k = pn + \alpha_2,\\
b\mb{E}Z_k^- + (1-b)\mb{E}Z_k^+ &= \mb{E}Z_k = pn + \alpha_2 + p\Delta + \alpha_2\Delta/n.
\end{align*}
Therefore
\begin{align*}
&\mu_{x,y} - (n+\Delta)(pn+\alpha_2) + n(pn+\alpha_2+p\Delta+\alpha_2\Delta/n)\\
&= (x-r(n+\Delta))(\mb{E}Y_k^--\mb{E}Y_k^+) - (y-bn)(\mb{E}Z_k^--\mb{E}Z_k^+)\\
&= (x-r(n+\Delta))\cdot2\sqrt{\frac{p(1-p)n}{\pi}} - (y-bn)\cdot2\sqrt{\frac{p(1-p)n}{\pi}} + O(n^{4/5})\\
&= 2\sqrt{\frac{p(1-p)n}{\pi}}((x-rn) - (y-bn)) + O(n^{4/5}).
\end{align*}
We deduce
\begin{align*}
&\mu_{x,y} - 2\sqrt{\frac{p(1-p)n}{\pi}}(x-y) + O(n^{4/5})\\
&= -2\sqrt{\frac{p(1-p)n}{\pi}}n\bigg(\frac{p(\Delta-1) + 1/2 + \alpha_0 - \alpha_2}{2\sqrt{\pi p(1-p)n}} - \frac{p(-\Delta-1) + 1/2 + \alpha_1 - \alpha_2}{2\sqrt{\pi p(1-p)n}}\bigg)\\
&= \bigg(\frac{-2p\Delta + \alpha_1-\alpha_0}{\pi}\bigg)n.\qedhere
\end{align*}
\end{proof}
We now make the following definitions.
\begin{itemize}
    \item Recall that $\sigma = \sqrt{(2-2/\pi)p(1-p)}n$.
    \item We have
    \begin{align*}
    r^\ast &= r^\ast(\alpha_0,\alpha_1,\alpha_2) = \frac{1}{2} + \frac{p(\Delta-1) + 1/2 + \alpha_0 - \alpha_2}{2\sqrt{\pi p(1-p)n}},\\
    b^\ast &= b^\ast(\alpha_0,\alpha_1,\alpha_2) = \frac{1}{2} + \frac{p(-\Delta-1) + 1/2 + \alpha_1 - \alpha_2}{2\sqrt{\pi p(1-p)n}},
    \end{align*}
    which are within $O(n^{-3/4})$ of $r,b$ by \cref{lem:chop-statistics}.
    \item We let
    \[\mu^\ast(\alpha_0,\alpha_1,\alpha_2,x,y) = \bigg(\frac{\alpha_1-\alpha_0-2p\Delta}{\pi}\bigg)n + 2\sqrt{\frac{p(1-p)n}{\pi}}(x-y),\]
    which satisfies $\mu_{x,y} = \mu^\ast(\alpha_0,\alpha_1,\alpha_2,x,y) + O(n^{4/5})$ for $|x-n/2|,|y-n/2|\le\sqrt{n}\log n$ by \cref{clm:mu-x-y}.
\end{itemize}

Now, continuing \eqref{eq:B-B-cond}, we find for $|x-n/2|,|y-n/2|\le\sqrt{n}\log n$ that
\begin{align}
&\mb{P}_{\mc{B}_{q_0}^{n+\Delta},\mc{B}_{q_1}^n,\mc{B}_{q_2}^{n+\Delta,n}}[|R_0\cap R_1| = x\wedge|B_0\cap B_1| = y\wedge|\mbf{S}| = |\mbf{T}|]\notag\\
&= \mb{P}[|R_0\cap R_1| = x]\mb{P}[|B_0\cap B_1| = y]\bigg(\frac{1}{\sqrt{2\pi}\sigma}\exp\bigg(\frac{\mu_{x,y}^2}{2\sigma^2}\bigg) + O(n^{-1/5}\sigma^{-1})\bigg)\notag\\
&= \phi_{rn,r(1-r)n}(x)\phi_{bn,b(1-b)n}(y)\phi_{0,\sigma^2}(\mu_{x,y}) + O(n^{-2-1/5})\notag\\
&= \phi_{r^\ast n,n/4}(x)\phi_{b^\ast n,n/4}(y)\phi_{0,\sigma^2}(\mu^\ast)+ O(n^{-2-1/5})\label{eq:B-Gaussian}
\end{align}
where $\phi_{a,b}$ denotes the pdf of the Gaussian with mean $a$ and variance $b$. In the second line we used that the conditional probability in \eqref{eq:B-B-cond} given $R_0\cap R_1$ and $B_0\cap B_1$ depends only on their sizes. The third line used a local limit theorem for binomials and appropriately expanding out error terms. The fourth line is just manipulation of established error terms in ways that we have seen already. Note that this equality is actually true if either $x$ or $y$ deviates by at least $\sqrt{n}\log n$ from $n/2$ as then the probability $|R_0\cap R_1| = x$ and $|B_0\cap B_1| = y$ is super-polynomially small. Therefore, this equation is true in general.

It is also worth mentioning by similar logic that if either $|x-n/2|\ge C\sqrt{\log n}$ or $|y-n/2|\ge C\sqrt{\log n}$ then
\begin{equation}\label{eq:B-sub-Gaussian}
\mb{P}_{\mc{B}_{q_0}^{n+\Delta},\mc{B}_{q_1}^n,\mc{B}_{q_2}^{n+\Delta,n}}[|R_0\cap R_1| = x\wedge|B_0\cap B_1| = y\wedge|\mbf{S}| = |\mbf{T}|]\le\mb{P}_{\mc{B}}[|R_0\cap R_1| = x] = O(n^{-10}).
\end{equation}

\subsection{Putting it together}\label{sub:integration}
Finally, note that the denominator of \cref{eq:B-even-B} is the probability that two samples of $\on{Bin}(n(n+\Delta),q_2)$ subtract to $0$. This satisfies a local limit theorem (e.g.~by \cite{Can80}) and has mean $0$ and variance $2q_2(1-q_2)n(n+\Delta) = 2p(1-p)n^2(1+O(n^{-1/2}))$, so
\begin{equation}\label{eq:B-denominator}
\mb{P}_{\mc{B}_{q_2}^{n+\Delta,n}}[|\mbf{S}| = |\mbf{T}|] = \frac{1}{2\sqrt{\pi p(1-p)}n} + O(n^{-5/4}).
\end{equation}

Putting together \cref{eq:D-I}, \cref{eq:I-E}, \cref{eq:E-B-even}, \cref{eq:B-even-B}, and \cref{eq:B-Gaussian} along with \cref{eq:B-denominator}, we obtain for some absolute $c > 0$ that
\begin{align*}
&\mb{P}_{\substack{\mc{D}_p^{n+\Delta},\mc{D}_p^n\\\mc{D}_p^{n+\Delta,n}}}[|R_0\cap R_1| = x\wedge|B_0\cap B_1| = y]\\
&= \int_{|\alpha_0|,|\alpha_1|,|\alpha_2|\le 20\sqrt{p(1-p)\log n}}\frac{\phi_{r^\ast n,n/4}(x)\phi_{b^\ast n,n/4}(y)\phi_{0,\sigma^2}(\mu^\ast)}{1/(2\sqrt{\pi p(1-p)}n)}d\nu(\alpha_0,\alpha_1,\alpha_2)+ O(n^{-1-c}).\notag\\
&= \int_{\alpha_0,\alpha_1,\alpha_2\in\mb{R}}\frac{\phi_{r^\ast n,n/4}(x)\phi_{b^\ast n,n/4}(y)\phi_{0,\sigma^2}(\mu^\ast)}{1/(2\sqrt{\pi p(1-p)}n)}d\nu(\alpha_0,\alpha_1,\alpha_2)+ O(n^{-1-c}).\notag
\end{align*}
where $\nu$ denotes the product measure of three independent Gaussians centered at $0$ with variances $p(1-p), p(1-p), p(1-p)/2$, respectively. Note that the difference between sampling the $\alpha$ values from $\nu$ or the $q$ values from $\mu$ is negligible.

Equivalently, we can sample $\beta_i = \alpha_i/\sqrt{p(1-p)}$ from Gaussians with variances $1,1,1/2$ for $i=0,1,2$, respectively. We have
\begin{align*}
&\int_{\alpha_0,\alpha_1,\alpha_2\in\mb{R}}\frac{\phi_{r^\ast n,n/4}(x)\phi_{b^\ast n,n/4}(y)\phi_{0,\sigma^2}(\mu^\ast)}{1/(2\sqrt{\pi p(1-p)}n)}d\nu(\alpha_0,\alpha_1,\alpha_2)\\
&= \frac{2\sqrt{\pi p(1-p)}n}{(\sqrt{2\pi})^5(\sqrt{\pi})(n/4)\sqrt{(2-2/\pi)p(1-p)}n}\int_\beta e^{-\frac{2(x-r^\ast n)^2+2(y-b^\ast n)^2}{n}-\frac{(\mu^\ast)^2}{(4-4/\pi)p(1-p)n^2}-\frac{\beta_0^2+\beta_1^2+2\beta_2^2}{2}}d\beta.
\end{align*}
When the values of $r^\ast,b^\ast,\mu^\ast$ are substituted in, this becomes a Gaussian integral in $\beta_0,\beta_1,\beta_2$.

Let
\begin{align*}
\sqrt{\frac{n}{4\pi}}x' &= x - \bigg(\frac{1}{2}+\frac{p(\Delta-1)+1/2}{2\sqrt{\pi p(1-p)n}}\bigg)n,\\
\sqrt{\frac{n}{4\pi}}y' &= y - \bigg(\frac{1}{2}+\frac{p(-\Delta-1)+1/2}{2\sqrt{\pi p(1-p)n}}\bigg)n.
\end{align*}
Then we deduce
\begin{align*}
&-\frac{2(x-r^\ast n)^2+2(y-b^\ast n)^2}{n}-\frac{(\mu^\ast)^2}{(4-4/\pi)p(1-p)n^2}-\frac{\beta_0^2+\beta_1^2+2\beta_2^2}{2}\\
&= -\frac{1}{2\pi p(1-p)}(x'\sqrt{p(1-p)}-\alpha_0+\alpha_2)^2 - \frac{1}{2\pi p(1-p)}(y'\sqrt{p(1-p)}-\alpha_1+\alpha_2)^2 \\
&\quad-\frac{1}{4\pi(\pi-1)p(1-p)}((x'-y')\sqrt{p(1-p)}+\alpha_1-\alpha_0)^2-\frac{\beta_0^2+\beta_1^2+2\beta_2^2}{2}\\
&= -\frac{1}{2\pi}(x'-\beta_0+\beta_2)^2-\frac{1}{2\pi}(y'-\beta_1+\beta_2)^2-\frac{1}{4\pi(\pi-1)}(x'-y'+\beta_1-\beta_0)^2-\frac{\beta_0^2+\beta_1^2+2\beta_2^2}{2}.
\end{align*}
Changing variables via $\alpha_i = \sqrt{p(1-p)}\beta_i$ therefore yields
\begin{align*}
&\int_{\alpha_0,\alpha_1,\alpha_2\in\mb{R}}\frac{\phi_{r^\ast n,n/4}(x)\phi_{b^\ast n,n/4}(y)\phi_{0,\sigma^2}(\mu^\ast)}{1/(2\sqrt{\pi p(1-p)}n)}d\nu(\alpha_0,\alpha_1,\alpha_2)\\
&= \frac{1}{\pi^2\sqrt{\pi-1}n}\int_\beta e^{-\frac{1}{2\pi}(x'-\beta_0+\beta_2)^2-\frac{1}{2\pi}(y'-\beta_1+\beta_2)^2-\frac{1}{4\pi(\pi-1)}(x'-y'+\beta_1-\beta_0)^2-\frac{\beta_0^2+\beta_1^2+2\beta_2^2}{2}}d\beta\\
&= \frac{2}{n\sqrt{\pi(2+\pi)}}\exp\bigg(-\frac{(1+\pi)(x')^2-2(x'y')+(1+\pi)(y')^2}{2\pi(2+\pi)}\bigg).
\end{align*}

Furthermore, if either $|x-n/2|\ge C\sqrt{\log n}$ or $|y-n/2|\ge C\sqrt{\log n}$ for appropriate $C > 0$ then we obtain a bound of size $O(n^{-5})$, which is easily seen using \cref{eq:B-sub-Gaussian} along with \cref{sub:numerator}. \cref{eq:D-I}, \cref{eq:I-E}, \cref{eq:E-B-even}, \cref{eq:B-even-B}, and \cref{eq:B-denominator}. This completes the proof of \cref{thm:day-one}.

\section{Tracking the remainder}\label{sec:days}
Now we adapt the approach of Ferber, Kwan, Narayanan, and the authors \cite{FKNSS21} to analyze the remainder of the majority dynamics process. Note that it is key that we computed what the leads were after day one at the scale of $\sqrt{n}$, since the techniques in that work only constrain objects at the scale $O(n^{1-\eta})$. However, essentially the same set of coarse data that is tracked in that work, along with the information from \cref{thm:day-one}, will allow us to perform an analysis of the remaining process via iterated revelation.

\subsection{Tracking degree parameters}\label{sub:tracking-data}
We first define the parameters that will be tracked, which are basically the joint degree distributions of each part of the graph to each of the other parts.

Given $k\ge 1$ and $x\in\{0,1\}^k$ of the form $(x_0,\ldots,x_{k-1})$, let $V_x = \cap_{i=0}^{k-1}X_i(x)$ where $X_i(0) = R_i$ and $X_i(1) = B_i$. Additionally, for $v\in B_0\cup R_0$ let
\[\on{deg}^{(k)} v = ((\deg_{V_x} v-p|V_x|)/\sqrt{p(1-p)n})_{x\in\{0,1\}^k}.\]
Finally, for $x\in\{0,1\}^k$ let $\mc{L}_x$ be the distribution of $\on{deg}^{(k)}v$ if we sample a uniform $v\in V_x$ (implicitly assuming it is nonempty).

It will be helpful to recall the following definition of Kolmogorov distance.
\begin{definition}\label{def:kolmogorov}
If $\mc{L}$ and $\mc{L}'$ are probability distributions on $\mb{R}^d$, the \emph{Kolmogorov distance} $\on{d}_{\mr{K}}(\mc{L},\mc{L}')$ is the supremum of
$|\mc{L}(A)-\mc{L}'(A)|$ over all sets $A = (-\infty,a_1]\times \dots\times (-\infty,a_d]$, where $a_1, \dots, a_d \in \mb{R}$.
\end{definition}

\subsection{Additional data from day one}\label{sub:day-one}
We now quickly derive certain coarse degree statistics arising from day one. These results are substantially less delicate than the previous section.
\begin{lemma}\label{lem:coarse-day-one}
There are $C, c > 0$ such that the following holds. Let $n\ge 2$, let $0\le\Delta\le(\log n)^{1/4}$ and let $(\log n)^{-1/4}\le p\le 1-(\log n)^{-1/4}$. For each $x\in\{0,1\}$ we have $\on{d}_{\rm{K}}(\mc{L}_x,\mc{N}(0,I_2))\le n^{-c}$ with probability $1-O(n^{-5})$ under majority dynamics on $\mb{G}(2n+\Delta,p)$ with $|R_0| = n+\Delta$. Furthermore, $\mc{L}_x$ is supported on $[-C\sqrt{\log n},C\sqrt{\log n}]$ with probability $1-O(n^{-5})$. (Hence we can choose $\wh{\mc{L}}_x$ to have the same support.)
\end{lemma}
\begin{remark}
A version of the above result when $p = 1/2$ and $\Delta = 0$ with a weaker probability bound appears in \cite[Section~4.1]{FKNSS21}, which is also sufficient for our purposes.
\end{remark}
\begin{proof}[Sketch]
It suffices to check it for $x = 0$, as the remaining case is analogous. The support claim is immediate by a union bound over all vertices. The Kolmogorov distance claim follows from the degree models in \cref{sec:one}. Specifically, consider the reduction from the true degree sequence model to the independent degree model, and then note that regardless of the revealed $q_0,q_1,q_2$ in $p\pm 20\sqrt{p(1-p)\log n}/n$, each vertex in $R_0$ has joint degree distribution extremely close to a correctly normalized Gaussian. Everything is now independent, so Chernoff on the number of vertices with degrees $pn+\sqrt{p(1-p)n}[x,x+n^{-c}]$, ranging over a polynomial-sized set of values $x$ proves the desired result. We must divide by the probability that the number of edges within each part is even and that the number of edges in the bipartite part agrees across both sides, as in \cref{sub:evenness,sub:numerator}, but these are polynomial probabilities which do not affect the bound significantly.
\end{proof}

\subsection{Data from day two}\label{sub:day-two}
We are now in position to derive the necessary data for day two. We show that given a substantial lead after day one that this leads grows to a linear size on the following day with high probability. To do this we reveal certain information and condition on certain high probability outcomes.
\begin{enumerate}[{\bfseries{E\arabic{enumi}}}]
\item\label{item:E1} Reveal all $\deg_{B_0}v$ and $\deg_{R_0}v$ values, which is enough to execute day one and determine $R_1,B_1$.
\item\label{item:E2} Furthermore, we assume that this revelation satisfies \cref{lem:coarse-day-one} and we let $|R_0\cap R_1| = x$ and $|B_0\cap B_1| = y$, defining
\begin{align*}
\sqrt{\frac{n}{4\pi}}x' &= x - \bigg(\frac{1}{2}+\frac{p(\Delta-1)+1/2}{2\sqrt{\pi p(1-p)n}}\bigg)n,\\
\sqrt{\frac{n}{4\pi}}y' &= y - \bigg(\frac{1}{2}+\frac{p(-\Delta-1)+1/2}{2\sqrt{\pi p(1-p)n}}\bigg)n.
\end{align*}
as in the statement of \cref{thm:day-one}.
\item\label{item:E3} We may assume that our revelation gave rise to values $x',y' = O(\sqrt{\log n})$ with probability at least $1-n^{-5}$ by \cref{thm:day-one}.
\item\label{item:E4} Finally, the number of edges between the two parts in the initial partition is $pn(n+\Delta)+O(n^{3/2-1/5})$ with super-polynomially high probability, so we may assume that our revelation gave rise to such a number of edges. Similarly within each part, we may assume we have $p\binom{n}{2} + O(n^{3/2-1/5})$ edges.
\end{enumerate}

In order to execute day two, we reveal $\deg_Tv$ for $T\in\{R_0\cap R_1,R_0\cap B_1, B_0\cap R_1, B_0\cap B_1\}$. Depending on the total degree from $v$ to $R_1$ and $B_1$, as well as whether $v\in B_1$ or $R_1$ in the case of ties, we know where $v$ lands in the next step.

\begin{claim}\label{clm:concentration}
There is an absolute $c > 0$ such that the following holds. Given revelations and assumptions \cref{item:E1,item:E2,item:E3,item:E4}, over the remaining randomness for each $x\in\{0,1\}^3$ the number of vertices in $V_x$ is concentrated at a scale $O(n^{1-c})$. In particular, $V_x$ is in an interval of length $O(n^{1-c})$ around its mean with probability at least $1-n^{-c}$.
\end{claim}
This implies that $|R_2|$ and $|B_2|$ are concentrated.

To do this, we attempt to understand the degree distribution better. First, let
\[\beta_i^R = \frac{\deg_{R_0}v_i-p(n+\Delta-1)}{\sqrt{p(1-p)(n+\Delta-1)}},\qquad\beta_i^B = \frac{\deg_{B_0}v_i-pn}{\sqrt{p(1-p)n}}\]
for $i\in R_0$ and
\[\beta_j^R = \frac{\deg_{R_0}v_j-p(n+\Delta)}{\sqrt{p(1-p)(n+\Delta)}},\qquad\beta_j^B = \frac{\deg_{B_0}v_j-p(n-1)}{\sqrt{p(1-p)(n-1)}}\]
for $j\in B_0$.

Given $v\in R_0\cup B_0$, look at
\[\deg^{(1)}v = (\alpha_0,\alpha_1),\qquad\deg^{(2)} v = (\rho_{00},\rho_{01},\rho_{10},\rho_{11})\]
where $\alpha_i = (\deg_{V_i}v-p|V_i|)/\sqrt{p(1-p)n}$ and $\rho_{ij} = (\deg_{V_{ij}}v-p|V_{ij}|)/\sqrt{p(1-p)n}$. Note that $\alpha_0,\alpha_1$ are determined given the revealed information, and that
\[\rho_{00}+\rho_{01} = \alpha_0,\qquad\rho_{10}+\rho_{11} = \alpha_1\]
hold. Furthermore $\rho_{00}$ and $\rho_{01}$ are independent given the revealed information, and their probability distributions can be determined by \cref{sub:enum-graph-model} and \cref{sub:enum-bigraph-model}, respectively.

By the first parts of \cref{prop:graph-bounded} and \cref{prop:bigraph-bounded}, we see that with super-polynomially high probability the $\rho_{ij}$ are bounded by $(\log n)^{25}/\sqrt{p(1-p)}$. Therefore we may assume that all vertices satisfy such a bound when revealing the new joint distribution of degrees. Furthermore, using the conditions on $x',y'$, in both cases we will be able to apply \cref{prop:graph-balanced} or \cref{prop:bigraph-balanced}, as long as we verify the necessary condition regarding $\sum\beta$ (in the notation of those propositions). Specifically, one needs for $v\in R_0$ that
\[\sum_{i\in R_0\setminus v}\beta_i^R = O(n^{5/6}) = \sum_{j\in B_0}\beta_j^R\]
while for $v\in B_0$ one needs
\[\sum_{i\in R_0}\beta_i^B = O(n^{5/6}) = \sum_{j\in B_0\setminus v}\beta_j^B.\]
This follows since we assumed the number of edges between the two parts in the initial partition is $pn(n+\Delta)+O(n^{3/2-1/5})$, and similar for within each of the two parts.

Now \cref{prop:graph-balanced} and \cref{prop:bigraph-balanced} show for $v\in R_0$ that
\begin{align}
\mb{P}[\rho_{00} = \gamma] &= \frac{\sqrt{2}+O(n^{-1/10})}{\sqrt{\pi p(1-p)n}}\exp\bigg(-\frac{1}{2}\bigg(2\gamma-\alpha_0-\frac{\sum_{i\in V_{00}}\beta_i^R}{n/2}\bigg)^2\bigg),\notag\\
\mb{P}[\rho_{10} = \gamma] &= \frac{\sqrt{2}+O(n^{-1/10})}{\sqrt{\pi p(1-p)n}}\exp\bigg(-\frac{1}{2}\bigg(2\gamma-\alpha_1-\frac{\sum_{j\in V_{10}}\beta_j^R}{n/2}\bigg)^2\bigg),\label{eq:day-two-gaussian-R}
\end{align}
absorbing negligible errors such as the difference of $\Delta$ between the number of vertices of $R_0$ and $B_0$.

Note that if we reveal the neighborhood of $v$, then any $w\in R_0$ will have essentially the same conditional distribution (the effect of revealing this neighborhood is to slightly adjust some degrees, which negligibly affects $\sum_{j\in R_{10}}\beta_j$, for instance).

Using this observation, a second-moment computation demonstrates that the number of vertices $v\in R_0\cap R_1$ with
\begin{equation}\label{eq:switch-condition-2}
\rho_{00} + \rho_{10} - \rho_{01} - \rho_{11}\ge\frac{p}{\sqrt{p(1-p)n}}(|V_{01}|+|V_{11}|-|V_{00}|-|V_{10}|) = \frac{p}{\sqrt{p(1-p)n}}(|B_1|-|R_1|)
\end{equation}
is concentrated (which corresponds to $v\in R_0\cap R_1$ being in $R_2$ after day two is revealed). We forgo the computational details (for similar arguments of this form, see \cite[Section~4.3.6]{FKNSS21}). The other cases are analogous. This completes the justification of \cref{clm:concentration}.

We quickly record that for $v\in B_0$, one obtains instead
\begin{align}
\mb{P}[\rho_{00} = \gamma] &= \frac{\sqrt{2}+O(n^{-1/10})}{\sqrt{\pi p(1-p)n}}\exp\bigg(-\frac{1}{2}\bigg(2\gamma-\alpha_0-\frac{\sum_{i\in V_{00}}\beta_i^B}{n/2}\bigg)^2\bigg),\notag\\
\mb{P}[\rho_{10} = \gamma] &= \frac{\sqrt{2}+O(n^{-1/10})}{\sqrt{\pi p(1-p)n}}\exp\bigg(-\frac{1}{2}\bigg(2\gamma-\alpha_1-\frac{\sum_{j\in V_{10}}\beta_j^B}{n/2}\bigg)^2\bigg).\label{eq:day-two-gaussian-B}
\end{align}

\begin{claim}\label{clm:expectation}
There is an absolute $c > 0$ such that the following holds. Given revelations and assumptions \cref{item:E1,item:E2,item:E3,item:E4}, over the remaining randomness we have
\[\mb{E}|R_0\cap R_2| = n\mb{P}_{Z\sim\mc{N}(0,2)}\bigg[Z\ge\frac{2}{\sqrt{\pi}} + \sqrt{\frac{p}{(1-p)n}}(|B_1|-|R_1|)\bigg] + O(n^{1-c})\]
and
\[\mb{E}|B_0\cap R_2| = n\mb{P}_{Z\sim\mc{N}(0,2)}\bigg[Z\ge -\frac{2}{\sqrt{\pi}} + \sqrt{\frac{p}{(1-p)n}}(|B_1|-|R_1|)\bigg] + O(n^{1-c}),\]
while also $\mb{E}|V_{x_1x_2x_3}| = \mb{E}|V_{x_1x_2'x_3}| + O(n^{1-c})$ for all $x_1,x_1',x_2,x_3\in\{0,1\}$.
\end{claim}
For this we note that if $v$ has parameters $(\alpha_0,\alpha_1)$ defined above then
\[\rho_{00}+\rho_{10}-\rho_{01}-\rho_{11} = (2\rho_{00}-\alpha_0) + (2\rho_{10}-\alpha_1),\]
which is the sum of two independent discrete Gaussians of standard deviation $1$ and discretization $2/\sqrt{p(1-p)n}$ by \cref{eq:day-two-gaussian-R,eq:day-two-gaussian-B}. A simple computation shows the sum of two such discrete Gaussians with the given error terms (and tail bounds) is a corresponding discrete Gaussian. In particular, we see for $v\in R_0$ that
\[\mb{P}[(2\rho_{00}-\alpha_0) + (2\rho_{10}-\alpha_1) = \tau] = \frac{1}{\sqrt{p(1-p)n}}\frac{1}{\sqrt{4\pi}}\exp\bigg(-\frac{1}{4}\bigg(\tau-\frac{\sum_{i\in V_{00}\cup V_{10}}\beta_i^R}{n/2}\bigg)^2\bigg) + O(n^{-1/2-1/12})\]
for $\tau$ on an appropriate integer lattice of discretization $2/\sqrt{p(1-p)n}$. A similar formula with $\beta_i^B$ holds for $v\in B_0$. For any $v\in R_0$ we see that
\[\mb{P}[\cref{eq:switch-condition-2}\text{ for }v] = \mb{P}_{Z\sim\mc{N}(0,2)}\bigg[Z\ge\frac{\sum_{i\in V_{00}\cup V_{10}}\beta_i^R}{n/2} + \sqrt{\frac{p}{(1-p)n}}(|B_1|-|R_1|)\bigg]+O(n^{-1/13})\]
hence
\[\mb{E}|R_0\cap R_2| = n\mb{P}_{Z\sim\mc{N}(0,2)}\bigg[Z\ge\frac{\sum_{i\in V_{00}\cup V_{10}}\beta_i^R}{n/2} + \sqrt{\frac{p}{(1-p)n}}(|B_1|-|R_1|)\bigg] + O(n^{12/13}).\]
Here the error term comes from the earlier term, as well as the possibility of vertices that are exactly balanced (of which there are few by the given computations) which may go a different way depending on its day one (not day zero) affiliation.

Similarly, for any $v\in B_0$ we have
\[\mb{P}[\cref{eq:switch-condition-2}\text{ for }v] = \mb{P}_{Z\sim\mc{N}(0,2)}\bigg[Z\ge\frac{\sum_{i\in V_{00}\cup V_{10}}\beta_i^B}{n/2} + \sqrt{\frac{p}{(1-p)n}}(|B_1|-|R_1|)\bigg]\]
and thus
\[\mb{E}|B_0\cap R_2| = n\mb{P}_{Z\sim\mc{N}(0,2)}\bigg[Z\ge\frac{\sum_{i\in V_{00}\cup V_{10}}\beta_i^B}{n/2} + \sqrt{\frac{p}{(1-p)n}}(|B_1|-|R_1|)\bigg] + O(n^{12/13}).\]
Finally, it suffices to compute the average of these $\beta_i^R$ and $\beta_i^B$ quantities over $V_{00}\cup V_{10} = R_1$. From \cref{lem:coarse-day-one} we know that the empirical normalized joint degree distributions for $R_0$ and $B_0$ are close to $\mc{N}(0,I_2)$. Therefore the degree distribution of $R_1$ is close to that of $(Z_1,Z_2)\sim\mc{N}(0,I_2)$ conditional on $Z_1\ge Z_2$ (where $Z_1$ corresponds to the parameter $(\deg_{R_0}v-p(n+\Delta))/\sqrt{p(1-p)(n+\Delta)}$). Therefore the expected value of this ensemble is within $O(n^{-c})$ of
\[\mb{E}[Z_1|Z_1\ge Z_2] = \frac{1}{\sqrt{\pi}}.\]
This shows
\[\frac{\sum_{i\in V_{00}\cup V_{10}}\beta_i^R}{n} = \frac{1}{\sqrt{\pi}} + O(n^{-c}) = \frac{\sum_{i\in V_{00}\cup V_{10}}\beta_i^B}{n}.\]
Then the first part of \cref{clm:expectation} follows.

The second part of \cref{clm:expectation} follows from that fact that all of the expressions for probabilities above based on $v$ are independent of the values $(\alpha_0,\alpha_1)$, so we can sum over $v\in R_0\cap R_1$, for instance, by just summing over those $v\in R_0$ which have $\alpha_0 > \alpha_1$, of which there are $n/2 + O(\sqrt{n}(\log n))$ by \cref{item:E2}.

Putting \cref{clm:concentration,clm:expectation} together and simplifying the sum of the expectations in \cref{clm:expectation}, we obtain the following information about the distribution of the sizes after day two.
\begin{theorem}\label{thm:day-two}
There is an absolute $c > 0$ such that the following holds. Given revelations and assumptions \cref{item:E1,item:E2,item:E3,item:E4}, over the remaining randomness we have
\[|R_2| = n+n\int_{-\eta}^\eta\frac{1}{\sqrt{2\pi}}\exp\bigg(-\frac{1}{2}\Big(u-\sqrt{\frac{2}{\pi}}\bigg)^2\Big)du + O(n^{1-c})\]
with probability at least $1-n^{-c}$, where $\eta = \sqrt{p/(2(1-p)n)}(|R_1|-|B_1|)$.
\end{theorem}
\begin{remark}
The integral is a signed integral. In particular its sign is the same as $\eta$.
\end{remark}

\subsection{Finishing on day three}\label{sub:day-three}
To finish we now use rather coarse consequences of degree enumeration to prove that every vertex is of the appropriate color. Note that \cref{thm:day-one} tells us the distribution of the lead $|R_1|-|B_1|$, and \cref{thm:day-two} tells us, in terms of the lead after day one, what the lead after day two is concentrated at. Furthermore, if the lead at day one is sufficiently positive then so will be the lead at day two with high probability. For example, for $p\le 1/2$ a lead of $\sqrt{n}$ yields a lead of $\Omega(n\sqrt{p})$.

Using arguments in \cite{BCOTT16,TV20,BD20} one can immediately prove that the side leading after day two has colored all the vertices in two further days. This along with \cref{thm:day-one,thm:day-two} will immediately justify \cref{thm:main} except that we can only guarantee it ends by day four. The arguments in \cite{BCOTT16,TV20,BD20} appear not sufficiently refined to deliver the day three result.
\begin{proof}[Proof of \cref{thm:main}]
Make revelations as in \cref{item:E1,item:E2,item:E3,item:E4}. Then reveal the information $\deg^{(2)}v$ for all $v\in R_0\cup B_0$, which allows us to determine the parts up to the end of day two. Reveal such that \cref{clm:concentration,clm:expectation,thm:day-two} are satisfied. Let
\[\eta = \sqrt{\frac{p}{2(1-p)n}}(|R_1|-|B_1|)\]
and note
\[|R_2|-|B_2| = 2n\int_{-\eta}^\eta\frac{1}{\sqrt{2\pi}}\exp\bigg(-\frac{1}{2}\Big(u-\sqrt{\frac{2}{\pi}}\bigg)^2\Big)du + O(n^{1-c})\]
from \cref{thm:day-two}, for some small absolute constant $c\in(0,1/4)$.

We wish to show that over all of the randomness (including the revealed randomness), if $\eta > 0$ then red will win in $3$ days while if $\eta < 0$ then red will win in $3$ days.

First if $\eta\ge 10+\sqrt{2\log p}$ then $|B_2|\le pn/5$. We see that red wins after day three with extremely high probability since the initial graph has minimum degree at least $pn/2$ with probability at least $1-O(\exp(-\Omega(pn)))$, and this forces every vertex to have more neighbors on the red side than blue side after day two is finished.

Similarly, if $-\eta\ge 10+\sqrt{2\log p}$ then blue wins after day three with extremely high probability.

The case $|\eta|\le n^{-c/2}$ occurs with probability $O(n^{-c/4})$ by the local limit theorem of \cref{thm:day-one}, so we ignore it.

Finally, without loss of generality we consider the case $n^{-c/2} < \eta < 10 + \sqrt{2\log p}$ (the opposite case being analogous except with red and blue switched).

Now to determine what happens on day three, we reveal $\deg^{(3)}v$ for all $v\in R_0\cup B_0$. This comes from another ensemble of degree-constrained distributions, so we apply the results of \cref{app:mckay-wormald} again, between all pairs of the four parts $V_x$ for $x\in\{0,1\}^2$. First, by \cref{clm:concentration,clm:expectation} we have
\[|V_x|\ge n\mb{P}_{Z\sim\mc{N}(0,2)}[Z\ge\frac{2}{\sqrt{\pi}}+\eta\sqrt{2}]+O(n^{1-c})\ge p^2n\ge n/(\log n)^{1/8}\]
for all $x\in\{0,1\}^3$. Thus we are in position to apply \cref{prop:graph-bounded,prop:bigraph-bounded} (as this guarantees the condition on the parameter $h$).

We need to check that the values $\beta$ are indeed of size $O((\log n)^2)$. This follows from the results of \cref{sub:day-two}; recall that we showed the $\rho_{ij}$ corresponding to each $v$ was bounded by $(\log n)^{25}/\sqrt{p(1-p)}$ with super-polynomially high probability, and otherwise there was an exact formula which guarantees the appropriate boundedness with super-polynomially high probability.

Now \cref{prop:graph-bounded,prop:bigraph-bounded} show that with super-polynomially high probability, for all $v\in R_0\cup B_0$ and $x\in\{0,1\}^3$ we have
\[\deg_{V_x}v = p|V_x| + O(n^{1/2}(\log n)^{25}).\]
Therefore
\[\deg_{R_2}v - \deg_{B_2}v = p(|R_2|-|B_2|) + O(n^{1/2}(\log n)^{25})\ge n^{1-c/3}\]
for all $v\in R_0\cup B_0$ with high probability. Thus every vertex will be on the red side after day three, as desired.

We have shown that with high probability (namely, as long as \cref{item:E1,item:E2,item:E3,item:E4} hold and $|\eta| > n^{-c/2}$, and over the randomness of certain degree revelations over three days), some color has the lead after the first day and it wins in three days. This probability is in fact polynomially good.

Finally, \cref{thm:day-one} tells us the probability that $\eta > 0$ to a high degree of accuracy, and simple computation with normal distributions shows it is
\[\mb{P}_{Z\sim\mc{N}(0,1)}\bigg[Z\le\frac{p\Delta\sqrt{2}}{\sqrt{\pi p(1-p)}}\bigg] + O(n^{-c})\]
if $c > 0$ is a small enough absolute constant. We are done.
\end{proof}

\bibliographystyle{amsplain0.bst}
\bibliography{main.bib}

\appendix
\section{Miscellaneous calculations}\label{app:calculations}
We now include a proof of \cref{lem:chop-probability}.
\begin{proof}[Proof of \cref{lem:chop-probability}]
We have
\begin{align*}
&\mb{P}[\on{Bin}(n+\tau,q)\ge\on{Bin}(n,q')]\\
&= \sum_{i=0}^n\sum_{j = i}^{n+\tau}\binom{n+\tau}{j}\bigg(p+\frac{\alpha}{n}\bigg)^j\bigg(1-p-\frac{\alpha}{n}\bigg)^{n+\tau-j}\binom{n}{i}\bigg(p+\frac{\beta}{n}\bigg)^i\bigg(1-p-\frac{\beta}{n}\bigg)^{n-i}\\
&= \sum_{i=0}^n\sum_{j = i}^{n+\tau}\binom{n+\tau}{j}p^j(1-p)^{n+\tau-j}\binom{n}{i}p^i(1-p)^{n-i}\\
&\qquad\qquad\qquad\times\bigg(1+\frac{\alpha}{pn}\bigg)^j\bigg(1-\frac{\alpha}{(1-p)n}\bigg)^{n+\tau-j}\bigg(1+\frac{\beta}{pn}\bigg)^i\bigg(1-\frac{\beta}{(1-p)n}\bigg)^{n-i}\\
&= \sum_{i=0}^n\sum_{j = i}^{n+\tau}\binom{n+\tau}{j}p^j(1-p)^{n+\tau-j}\binom{n}{i}p^i(1-p)^{n-i}\\
&\qquad\qquad\qquad\times\bigg(1+O\bigg(\frac{(\log n)^2}{n}\bigg)\bigg)\exp\bigg(\frac{j\alpha}{pn} - \frac{(n+\tau-j)\alpha}{(1-p)n} + \frac{i\beta}{pn} - \frac{(n-i)\beta}{(1-p)n}\bigg).
\end{align*}
The last line comes from the bounds on $\alpha,\beta,p$. Note that terms where $|i-pn|\ge C\sqrt{p(1-p)n\log n}$ or $|j-pn|\ge C\sqrt{p(1-p)n\log n}$ can contribute at most $O(1/n)$ to the total mass, if $C$ is a large enough constant. Let $i = pn + x\sqrt{p(1-p)n}$ and $j = pn + y\sqrt{p(1-p)n}$ where $x,y$ range over appropriate values of magnitude at most $C\sqrt{\log n}$. The exponential term in the above can be expanded if $i,j$ have the above bounded criterion, and otherwise the total contribution is negligible anyway. We thus find that the above equals
\begin{align*}
&\sum_{i=0}^n\sum_{j=i}^{n+\tau}\binom{n+\tau}{j}p^j(1-p)^{n+\tau-j}\binom{n}{i}p^i(1-p)^{n-i} + O(n^{-1}(\log n)^2)\\
&+ \sum_{i=0}^n\sum_{j=i}^{n+\tau}\bigg(\frac{j\alpha}{pn} - \frac{(n+\tau-j)\alpha}{(1-p)n} + \frac{i\beta}{pn} - \frac{(n-i)\beta}{(1-p)n}\bigg)\binom{n+\tau}{j}p^j(1-p)^{n+\tau-j}\binom{n}{i}p^i(1-p)^{n-i}\\
&= \mb{P}[\on{Bin}(n+\tau,p)\ge\on{Bin}(n,p)] + O(n^{-1}(\log n)^2)\\
&\qquad\qquad+ \sum_x\sum_{y\ge x}\bigg(\frac{y\alpha+x\beta}{\sqrt{p(1-p)n}}\bigg)\binom{n+\tau}{j}p^j(1-p)^{n+\tau-j}\binom{n}{i}p^i(1-p)^{n-i}\\
&= \mb{P}[\on{Bin}(n+\tau,p)\ge\on{Bin}(n,p)] + \frac{1}{\sqrt{p(1-p)n}}\int_{u=-\infty}^\infty\int_{v=u}^\infty (v\alpha+u\beta)\frac{1}{2\pi}\exp(-(u^2+v^2)/2)dudv\\
&\qquad\qquad+ O((\log n)^2n^{-1})\\
&= \mb{P}[\on{Bin}(n+\tau,p)\ge\on{Bin}(n,p)] + \frac{\mb{E}_{Z_1,Z_2\sim\mc{N}(0,1)}[\alpha Z_1 + \beta Z_2|Z_1\ge Z_2]}{2\sqrt{p(1-p)n}} + O((\log n)^2n^{-1})\\
&= \mb{P}[\on{Bin}(n+\tau,p)\ge\on{Bin}(n,p)] + \frac{\alpha-\beta}{2\sqrt{\pi p(1-p)n}} + O((\log n)^2n^{-1})
\end{align*}
In the second line we dropped the $\tau$ term as $(\tau\alpha)/((1-p)n)$ is negligible, and in the third line we used the multidimensional Berry-Esseen theorem. For the final line, we used
\[\mb{E}_{Z_1,Z_2\sim\mc{N}(0,1)}[Z_1-Z_2|Z_1\ge Z_2] = \mb{E}_{G\sim\mc{N}(0,2)}|G| = \frac{2}{\sqrt{\pi}}\]
and $\mb{E}[Z_1+Z_2|Z_1\ge Z_2] = \mb{E}[Z_1+Z_2] = 0$.

Now let $m = \min(n+\tau,n)$. We compute by a local central limit theorem for $|k|\le|\tau|$ that for two independent binomial samples,
\begin{align*}
\mb{P}_{X,X'\sim\on{Bin}(m,p)}[X' = X + k] &= \mb{P}_{X,X'}[X'-X=k] \\
&= (1+O(n^{-1/4}))\frac{1}{\sqrt{4\pi p(1-p)m}}\exp\bigg(-\frac{k^2}{4p(1-p)m}\bigg)\\
&= \frac{1}{\sqrt{4\pi p(1-p)n}} + O(n^{-3/4}).
\end{align*}
Thus if $0 < k\le|\tau|$ we have by symmetry that
\[2\mb{P}_{X,X'\sim\on{Bin}(m,p)}[X'\ge X + k] + (2k-1)\bigg(\frac{1}{\sqrt{4\pi p(1-p)n}} + O(n^{-3/4})\bigg) = 1\]
hence
\[\mb{P}_{X,X'\sim\on{Bin}(m,p)}[X'\ge X + k] = \frac{1}{2} - \frac{2k-1}{4\sqrt{\pi p(1-p)n}} + O(n^{-3/4})\]
for all $0 < k\le|\tau|$. In fact this holds for $|k|\le|\tau|$ by a similar argument. Therefore if $\tau\ge 0$ we have
\begin{align*}
\mb{P}[\on{Bin}(n+\tau,p)\ge\on{Bin}(n,p)] &= \mb{E}_{k\sim\on{Bin}(\tau,p)}[\mb{P}_{X,X'\sim\on{Bin}(m,p)}[X'\ge X - k]]\\ &= \mb{E}_{k\sim\on{Bin}(\tau,p)}\bigg[\frac{1}{2} + \frac{2k+1}{4\sqrt{\pi p(1-p)n}} + O(n^{-3/4})\bigg]\\
&= \frac{1}{2} + \frac{2p\tau+1}{4\sqrt{\pi p(1-p)n}} + O(n^{-3/4}).
\end{align*}
A similar argument shows the same formula for $\tau < 0$.

Putting it all together, we see
\[\mb{P}[\on{Bin}(n+\tau,q)\ge\on{Bin}(n,q')] = \frac{1}{2}+\frac{p\tau + 1/2 + \alpha - \beta}{2\sqrt{\pi p(1-p)n}} + O(n^{-3/4}).\qedhere\]
\end{proof}

\section{Probabilities in degree-constrained models}\label{app:mckay-wormald}
We will first require slight modifications of \cite[Propositions~A.1,~A.6]{FKNSS21}.

\subsection{Graph model}\label{sub:enum-graph-model}
First we compute the probability of having certain neighborhood sizes in a degree-constrained model of graphs. This result follows from a delicate but straightforward argument that utilizes graph enumeration results from \cite{MW90} (the bipartite model, discussed later, utilizes bipartite graph enumeration results from \cite{CGM08}). We state those results precisely in \cref{sub:enum-computations}.
\begin{proposition}\label{prop:graph-expectation}
There are $\epsilon_{\ref{prop:graph-expectation}}, C_{\ref{prop:graph-expectation}} > 0$ so the following holds. Let $n\ge C_{\ref{prop:graph-expectation}}$ and $(\log n)^{-1/4}\le p\le 1-(\log n)^{-1/4}$. Let $\mbf{d}\in E_n$ such that each $d_w=pn+O(\sqrt{p(1-p)}n^{1/2+\epsilon_{\ref{prop:graph-expectation}}})$. Let $G$ be a uniformly random graph on vertex set $W = \{v_1,\ldots,v_n\}$ with this degree sequence. Consider a size $h$ subset $V\subseteq W$ satisfying $\min(h,n-h)\ge n/(\log n)^{1/8}$, and an integer $t\in[0,d_n]$. For $w\in W$, define $\beta_w$ by $d_w = p(n-1)+\beta_w\sqrt{p(1-p)(n-1)}$. Then
\begin{align*}
&\mb{P}[\deg_V(v_n) = t]\\
&= (1+O(n^{-1/6}))\exp\bigg(\frac{(\sum_{i=1}^n\beta_i)(\sum_{i=1}^n\beta_i-2n\beta_n)}{2n^2}\bigg)\frac{\binom{h-\mbm{1}_V(v_n)}{t}\binom{n-h-\mbm{1}_{V^c}(v_n)}{d_n-t}}{\binom{n-1}{d_n}}\times\\
&\mb{E}_{\substack{S_1\sim\binom{V\setminus v_n}{t}\\S_2\sim\binom{V^c\setminus v_n}{d_n-t}\\S=S_1\cup S_2}}\exp\bigg(-\sqrt{\frac{p}{1-p}}\sum_{i\in W\setminus v_n}\bigg(-\frac{1-p}{p}\bigg)^{\mbm{1}_S(i)}\frac{\beta_i}{\sqrt{n-1}}-\frac{1}{2}\sum_{i\in W\setminus v_n}\bigg(\frac{1-p}{p}\bigg)^{2\mbm{1}_S(i)-1}\frac{\beta_i^2}{n-1}\bigg).
\end{align*}
Here $S_1,S_2$ are uniform over their respective domains.
\end{proposition}
We will defer the proof of this to \cref{sub:enum-computations,sub:enum-proofs}. We now turn to various consequences of this formula. To proceed, we will need to understand expressions as appearing in the right side of \cref{prop:graph-expectation}. To this end, we state the following general results about sums of random variables constrained to live on a slice.
\begin{lemma}\label{lem:slice-subgaussian}
Let $a_1,\ldots,a_n\in\mb{R}$ and let $X = \sum_{i=1}^na_i\xi_i$, where $\xi = (\xi_1,\ldots,\xi_n)$ is uniform on the subset of $\{0,1\}^n$ with sum $s$. Furthermore assume that $\eta^2 = \sum_{i=1}^na_i^2-(\sum_{i=1}^na_i)^2/n$. We have
\[\mb{P}[|X-\mb{E}X|\ge t]\le 2\exp(-t^2/(4\eta^2))\]
and
\[\mb{E}e^{X-\mb{E}X}\le 2e^{O(\eta^2)}.\]
\end{lemma}
\begin{proof}
The first part is by Azuma-Hoeffding (see~\cite[Lemma~2.2]{KST19} for a detailed proof). The second part follows from integrating the first (see~\cite[Proposition~2.5.2]{Ver18}).
\end{proof}
\begin{lemma}\label{lem:slice-moment}
Let $a_1,\ldots,a_n\in\mb{R}$ and let $X = \sum_{i=1}^na_i\xi_i$, where $\xi = (\xi_1,\ldots,\xi_n)$ is uniform on the subset of $\{0,1\}^n$ with sum $s$ such that $\min(s,n-s)\ge n(\log n)^{-2}$. Furthermore assume that $|a_i|\le n^{-1/2}(\log n)^2$ and $\eta^2 = \sum_{i=1}^na_i^2-(\sum_{i=1}^na_i)^2/n\le \sqrt{\log n}$. We have
\[\mb{E}e^X = \exp\bigg(\mb{E}X + \frac{1}{2}\on{Var}X + O(n^{-1/9})\bigg).\]
\end{lemma}
\begin{proof}
Let $\mu = \mb{E}X$ and $\sigma^2 = \on{Var}X$. Clearly
\begin{align*}
\sigma^2&=\sum_{i\neq j}a_ia_j(\mb{E}[\xi_i\xi_j]-\mb{E}\xi_i\mb{E}\xi_j) + \sum_i a_i^2(\mb{E}\xi_i^2-(\mb{E}\xi_i)^2)\\
&=\sum_{i\neq j}a_ia_j\bigg(\frac{s(s-1)}{n(n-1)}-\frac{s^2}{n^2}\bigg) + \sum_i a_i^2\bigg(\frac{s}{n}-\frac{s^2}{n^2}\bigg)=\frac{s(n-s)}{n(n-1)}\eta^2.
\end{align*}

First by \cref{lem:slice-subgaussian} we have
\[\mb{P}[|X-\mb{E}X|\ge t]\le 2\exp(-t^2/(4\eta^2))\]
for all $t\ge 0$. Now
\begin{align*}
\mb{E}e^{X-\mb{E}X} &= \int_{-\infty}^\infty e^t\mb{P}[X-\mb{E}X\ge t]dt=\int_{-\infty}^{8\eta\sqrt{\log n}}e^t\mb{P}[X-\mb{E}X\ge t]dt +O\bigg(\int_{4\eta\sqrt{\log n}}^\infty e^{t-t^2/(4\eta^2)}dt\bigg)\\
&=\int_{-\infty}^{8\eta\sqrt{\log n}}e^t\mb{P}[X-\mb{E}X\ge t]dt +O\bigg(\int_{8\eta\sqrt{\log n}}^\infty e^{-t^2/(8\eta^2)}dt\bigg)\\
&=\int_{-\infty}^{8\eta\sqrt{\log n}}e^t\mb{P}[X-\mb{E}X\ge t]dt+ O(n^{-4}).
\end{align*}
If $\sigma\le n^{-1/8}$, then $\eta$ is similarly bounded and we obtain an upper bound of the form $1+O(n^{-1/9})$. Combining with $\mb{E}e^X\ge e^{\mb{E}X}$, the result follows. The result follows. If $\sigma > n^{-1/8}$ then a combinatorial central limit theorem of Bolthausen \cite{Bol84} shows
\[\on{d}_{\mr{K}}(X-\mb{E}X,\mc{N}(0,\sigma^2)) = O\bigg(\sum_{i=1}^n |a_i|^3/\sigma^3\bigg) = O(n^{-2/17}).\]
This allows us the replace the integrand above with the CDF of a Gaussian, and we easily derive
\[\mb{E}e^{X-\mb{E}X} = e^{\frac{\sigma^2}{2}} + O(n^{-2/17})\cdot O(e^{\eta\sqrt{\log n}}) = \exp(\sigma^2/2 + O(n^{-1/9})).\qedhere\]
\end{proof}

We now use this information to explicitly compute the formula in \cref{prop:graph-expectation} under some slight additional hypotheses.
\begin{proposition}\label{prop:graph-bounded}
Assume the hypotheses of \cref{prop:graph-expectation}. Assume additionally that $\beta_w = O((\log n)^2)$ for all $w\in W$. Then if $|t-ph|\ge n^{1/2}(\log n)^{25}$ we have
\[\mb{P}[\deg_V(v_n) = t]\le\exp(-\Omega((t-ph)^2/n)).\]
If $|t-ph|\le n^{3/5}$ and furthermore $(\sum_{i=1}^n\beta_i^2)/n\le(\log n)^{1/9}$ then we have
\begin{align*}
&\mb{P}[\deg_V(v_n) = t]\\
&= (1+O(n^{-1/10}))\frac{\binom{h}{t}\binom{n-h-1}{d_n-t}}{\binom{n-1}{d_n}}\exp\bigg[\frac{(\sum_{i=1}^n\beta_i)(\sum_{i=1}^n\beta_i-2n\beta_n)}{2n^2}\\
&-\sqrt{\frac{p}{1-p}}\bigg(\sum_{i\in V\setminus v_n}\bigg(1-\frac{t}{ph}\bigg)\frac{\beta_i}{\sqrt{n-1}}+\sum_{i\in V^c\setminus v_n}\bigg(1-\frac{d_n-t}{p(n-h)}\bigg)\frac{\beta_i}{\sqrt{n-1}}\bigg)-\frac{1}{2}\sum_{i\in W\setminus v_n}\frac{\beta_i^2}{n-1}\\
&+\frac{1}{2nh}\sum_{i < j\in V\setminus v_n}(\beta_i-\beta_j)^2 + \frac{1}{2n(n-h)}\sum_{i < j\in V^c\setminus v_n}(\beta_i-\beta_j)^2\bigg].
\end{align*}
\end{proposition}
\begin{proof}
We apply \cref{prop:graph-expectation}. Let
\[X = -\sqrt{\frac{p}{1-p}}\sum_{i\in W\setminus v_n}\bigg(-\frac{1-p}{p}\bigg)^{\mbm{1}_S(i)}\frac{\beta_i}{\sqrt{n-1}}-\frac{1}{2}\sum_{i\in W\setminus v_n}\bigg(\frac{1-p}{p}\bigg)^{2\mbm{1}_S(i)-1}\frac{\beta_i^2}{n-1},\]
where $S_1\sim\binom{V\setminus v_n}{t}$ and $S_2\sim\binom{V^c\setminus v_n}{d_n-t}$. The point will be that in typical cases $X$ is a random variable with sub-Gaussian tails, and that it is converging to a Gaussian, which are together enough to compute its exponential moment. When $t$ is far from $ph$, we will instead obtain a tail bound. We have
\begin{align*}
&\mb{E}X\\
&= -\sqrt{\frac{p}{1-p}}\sum_{i\in V\setminus v_n}\bigg(\frac{t(-(1-p))}{hp}+\frac{h-t}{h}\bigg)\frac{\beta_i}{\sqrt{n-1}} + O(n^{-1/3})\\
&\quad-\sqrt{\frac{p}{1-p}}\sum_{i\in V^c\setminus v_n}\bigg(\frac{(d_n-t)(-(1-p))}{(n-h)p}+\frac{(n-h)-(d_n-t)}{n-h}\bigg)\frac{\beta_i}{\sqrt{n-1}} + O(n^{-1/3})\\
&\quad-\frac{1}{2}\sum_{i\in V\setminus v_n}\bigg(\frac{t(1-p)}{hp}+\frac{(h-t)p}{h(1-p)}\bigg)\frac{\beta_i^2}{n-1} + O(n^{-1/3})\\
&\quad-\frac{1}{2}\sum_{i\in V^c\setminus v_n}\bigg(\frac{(d_n-t)(1-p)}{(n-h)p}+\frac{((n-h)-(d_v-t))p}{(n-h)(1-p)}\bigg)\frac{\beta_i^2}{n-1} + O(n^{-1/3})\\
&=-\sqrt{\frac{p}{1-p}}\bigg(\sum_{i\in V\setminus v_n}\bigg(1-\frac{t}{ph}\bigg)\frac{\beta_i}{\sqrt{n-1}}+\sum_{i\in V^c\setminus v_n}\bigg(1-\frac{d_n-t}{p(n-h)}\bigg)\frac{\beta_i}{\sqrt{n-1}}\bigg)\\
&\quad-\frac{1}{2}\sum_{i\in W\setminus v_n}\frac{\beta_i^2}{n-1} + O(n^{-1/3} + |t-ph|(\log n)^2/n).
\end{align*}
The initial additive error terms $O(n^{-1/3})$ come from the fact that $v_n\in V$ or $v_n\in V^c$ slightly change the fractions listed above, but not by much.

At this point, if $|t-ph|\ge n^{3/5}$, we have
\[\frac{\binom{h-\mbm{1}_V(v_n)}{t}\binom{n-h-\mbm{1}_{V^c}(v_n)}{d_n-t}}{\binom{n-1}{d_n}}\le\exp(-\Omega((t-ph)^2/n))\]
by tail bounds for the hypergeometric distribution (see e.g.~\cite[Theorem~2.10]{JLR00}). The initial exponential term is bounded by $\exp(O((\log n)^4))$, and we are left with $\mb{E}\exp(X)$. Now \cref{lem:slice-subgaussian} demonstrates $\mb{E}\exp(X)\le\exp(\mb{E}X + O((\log n)^4))$ since the coefficient variance in $X$ is $O((\log n)^4/n)$ by the given conditions. But the above demonstrates
\[|\mb{E}X| = O\bigg(\frac{|t-ph|}{\sqrt{n}}(\log n)^2\bigg).\]
This immediately gives a bound of the claimed quality.

From now on we assume $|t-ph|\le n^{3/5}$. Note that the error term computed on $\mb{E}X$ is now of quality $O(n^{-1/3})$ uniformly. We next compute the variance of $X$. It is straightforward to see that $X$ and 
\[X' = -\sqrt{\frac{p}{1-p}}\sum_{i\in W\setminus v_n}\bigg(-\frac{1-p}{p}\bigg)^{\mbm{1}_S(i)}\frac{\beta_i}{\sqrt{n-1}}\]
have $|\on{Var}X-\on{Var}X'| = O(n^{-1/4})$. From the proof of \cref{lem:slice-moment}, we see
\begin{align*}
\on{Var}X &=\frac{1}{p(1-p)(n-1)}\bigg(\frac{t(h-t)}{h(h-1)}\frac{\sum_{i < j\in V\setminus v_n}(\beta_i-\beta_j)^2}{h}\\
&\qquad\qquad+ \frac{(d_n-t)((n-h)-(d_n-t))}{(n-h)(n-h-1)}\frac{\sum_{i < j\in V^c\setminus v_n}(\beta_i-\beta_j)^2}{n-h}\bigg) + O(n^{-1/4}),
\end{align*}
where we again use that the fraction $t/|V\setminus v_n|$ is close to $t/h$ regardless of if $v_n\in V$. Using $t = ph + O(n^{3/5})$ and $d_n = pn + O(\sqrt{p(1-p)n}(\log n)^2)$, we find
\[\on{Var}X =  \frac{1}{nh}\sum_{i < j\in V\setminus v_n}(\beta_i-\beta_j)^2 + \frac{1}{n(n-h)}\sum_{i < j\in V^c\setminus v_n}(\beta_i-\beta_j)^2+ O(n^{-1/4}).\]
Note that $\on{Var}X\le\sum_{i=1}^n\beta_i^2/\min(h,n-h) = O(n(\log n)^{1/9}/\min(h,n-h))$. Now apply \cref{lem:slice-moment} to the two slices defining $X$. Note that the condition $\eta^2\le\sqrt{\log n}$ follows from the inequalities $(n/h)(p(1-p))^{-1}(\log n)^{1/9} < \sqrt{\log n}$ and the relation between $\sigma^2,\eta^2$ in the proof of \cref{lem:slice-moment}. Therefore
\[\mb{E}e^X = \exp\Big(\mb{E}X + \frac{1}{2}\on{Var}X + O(n^{-1/10})\Big).\]
Finally, using \cref{prop:graph-expectation}, we obtain
\begin{align*}
&\mb{P}[\deg_V(v_n) = t]\\
&= (1+O(n^{-1/10}))\frac{\binom{h}{t}\binom{n-h-1}{d_n-t}}{\binom{n-1}{d_n}}\exp\bigg[\frac{(\sum_{i=1}^n\beta_i)(\sum_{i=1}^n\beta_i-2n\beta_n)}{2n^2}\\
&-\sqrt{\frac{p}{1-p}}\bigg(\sum_{i\in V\setminus v_n}\bigg(1-\frac{t}{ph}\bigg)\frac{\beta_i}{\sqrt{n-1}}+\sum_{i\in V^c\setminus v_n}\bigg(1-\frac{d_n-t}{p(n-h)}\bigg)\frac{\beta_i}{\sqrt{n-1}}\bigg)-\frac{1}{2}\sum_{i\in W\setminus v_n}\frac{\beta_i^2}{n-1}\\
&+\frac{1}{2nh}\sum_{i < j\in V\setminus v_n}(\beta_i-\beta_j)^2 + \frac{1}{2n(n-h)}\sum_{i < j\in V^c\setminus v_n}(\beta_i-\beta_j)^2\bigg].
\end{align*}
We used that the product of binomials changes by a small factor upon swapping between $v\in V$ and $v\in V^c$.

\end{proof}

Finally we note a massive simplification of this formula in the case when $h$ is near $n/2$ and the total number of edges is close to $p\binom{n}{2}$.
\begin{proposition}\label{prop:graph-balanced}
Assume the hypotheses of the second part of \cref{prop:graph-bounded}. Assume additionally that $|h-n/2| = O(\sqrt{n\log n})$ and $\sum_{i\in W}\beta_i = O(n^{5/6})$. Then for $|\gamma|\le n^{1/10}$ with $ph+\gamma\sqrt{p(1-p)n}\in\mb{Z}$, we have
\[\mb{P}[\deg_V(v_n) = ph+\gamma\sqrt{p(1-p)n}] = \frac{\sqrt{2}+O(n^{-1/10})}{\sqrt{\pi p(1-p)n}} \exp\bigg(-\frac{1}{2}\bigg(2\gamma-\beta_n-\frac{\sum_{i\in V}\beta_i}{n/2}\bigg)^2\bigg).\]
\end{proposition}
\begin{proof}
Apply \cref{prop:graph-bounded} to $t = ph + \gamma\sqrt{p(1-p)n}$. First, since $\sum_{i=1}^n\beta_i = O(n^{3/4})$ and $|\beta_i| = O((\log n)^2)$, we see that the initial exponential term is small. Next, we have from $d_n = p(n-1)+\beta_n\sqrt{p(1-p)(n-1)}$ and $|h-n/2| = O(\sqrt{n\log n})$ that
\begin{align*}
&-\sqrt{\frac{p}{1-p}}\bigg(\sum_{i\in V\setminus v_n}\bigg(1-\frac{t}{ph}\bigg)\frac{\beta_i}{\sqrt{n-1}}+\sum_{i\in V^c\setminus v_n}\bigg(1-\frac{d_n-t}{p(n-h)}\bigg)\frac{\beta_i}{\sqrt{n-1}}\bigg)\\
&= \sum_{i\in V\setminus v_n}\frac{\gamma\beta_i}{n/2} + \sum_{i\in V^c\setminus v_n}\frac{(\beta_n-\gamma)\beta_i}{n/2} + O(n^{-1/3})\\
&= \frac{2\gamma-\beta_n}{n/2}\sum_{i\in V\setminus v_n}\beta_i + O(n^{-1/8})
\end{align*}
Similarly, in the last three terms of the formula in \cref{prop:graph-bounded}, we can replace $h$ by $n/2$ and $n-1$ by $n$ in return for a negligible additive error. Therefore the terms in the exponential add up to
\begin{align*}
&\frac{2\gamma-\beta_n}{n/2}\sum_{i\in V\setminus v_n}\beta_i - \frac{1}{2}\sum_{i\in W\setminus v_n}\frac{\beta_i^2}{n} + \frac{1}{n^2}\sum_{i<j\in V\setminus v_n}(\beta_i-\beta_j)^2 + \frac{1}{n^2}\sum_{i<j\in V^c\setminus v_n}(\beta_i-\beta_j)^2 + O(n^{-1/8})\\
&= \frac{2\gamma-\beta_n}{n/2}\sum_{i\in V\setminus v_n}\beta_i + \frac{1}{n^2}\sum_{i<j\in V\setminus v_n}(-2\beta_i\beta_j) + \frac{1}{n^2}\sum_{i<j\in V^c\setminus v_n}(-2\beta_i\beta_j) + O(n^{-1/8})\\
&= \frac{2\gamma-\beta_n}{n/2}\sum_{i\in V\setminus v_n}\beta_i - \frac{1}{n^2}\bigg(\sum_{i\in V\setminus v_n}\beta_i\bigg)^2 - \frac{1}{n^2}\bigg(\sum_{i\in V^c\setminus v_n}\beta_i\bigg)^2 + O(n^{-1/8})\\
&= \frac{2\gamma-\beta_n}{n/2}\sum_{i\in V\setminus v_n}\beta_i - \frac{2}{n^2}\bigg(\sum_{i\in V\setminus v_n}\beta_i\bigg)^2 + O(n^{-1/8}).
\end{align*}
Furthermore, the ratio of binomial coefficients can be computed as follows. If $|a-pb| = O(b^{3/5})$ then by Stirling's formula,
\begin{align*}
&p^a(1-p)^{b-a}\binom{b}{a}\\
&= (1+O(b^{-1/2}))\frac{1}{\sqrt{2\pi p(1-p)b}}\Big(\frac{a}{pb}\Big)^{-a}\Big(\frac{b-a}{(1-p)b}\Big)^{-(b-a)}\\
&= (1+O(b^{-1/2}))\frac{1}{\sqrt{2\pi p(1-p)b}}\bigg(1+\frac{a-pb}{pb}\bigg)^{-a}\bigg(1-\frac{a-pb}{(1-p)b}\bigg)^{-(b-a)}\\
&= \frac{1+O(b^{-1/6})}{\sqrt{2\pi p(1-p)b}}\exp\bigg(-a\frac{a-pb}{pb}+\frac{1}{2}a\Big(\frac{a-pb}{pb}\Big)^2+(b-a)\frac{a-pb}{(1-p)b}+\frac{1}{2}(b-a)\Big(\frac{a-pb}{(1-p)b}\Big)^2\bigg)\\
&= \frac{1+O(b^{-1/6})}{\sqrt{2\pi p(1-p)b}}\exp\bigg(-a\frac{a-pb}{pb}+\frac{1}{2}pb\Big(\frac{a-pb}{pb}\Big)^2+(b-a)\frac{a-pb}{(1-p)b}+\frac{1}{2}(b-pb)\Big(\frac{a-pb}{(1-p)b}\Big)^2\bigg)\\
&= \frac{1+O(b^{-1/6})}{\sqrt{2\pi p(1-p)b}}\exp\bigg(-\frac{(a-pb)^2}{2p(1-p)b}\bigg),
\end{align*}
so that a local central limit theorem holds. This allows us to compute

\begin{align*}
&\frac{\binom{h}{t}\binom{n-h-1}{d_n-t}}{\binom{n-1}{d_n}} = \frac{p^t(1-p)^{h-t}\binom{h}{t}p^{d_n-t}(1-p)^{(n-h-1)-(d_n-t)}\binom{n-h-1}{d_n-t}}{p^{d_n}(1-p)^{(n-1)-d_n)}\binom{n-1}{d_n}}\\
&\qquad= \frac{\sqrt{2}+O(n^{-1/8})}{\sqrt{\pi p(1-p)n}}\exp\bigg(-\frac{(t-ph)^2}{2p(1-p)h}-\frac{((d_n-t)-p(n-h-1))^2}{2p(1-p)(n-h-1)}+\frac{(d_n-p(n-1))^2}{2p(1-p)(n-1)}\bigg)\\
&\qquad= \frac{\sqrt{2}+O(n^{-1/8})}{\sqrt{\pi p(1-p)n}}\exp\bigg(-\gamma^2-(\beta_n-\gamma)^2+\frac{\beta_n^2}{2}\bigg)\\
&\qquad= \frac{\sqrt{2}+O(n^{-1/8})}{\sqrt{\pi p(1-p)n}}\exp\bigg(-\frac{(2\gamma-\beta_n)^2}{2}\bigg).
\end{align*}
Putting it all together in \cref{prop:graph-bounded} we obtain the result, noting that $\beta_n$ is small so a difference of $\beta_n/(n/2)$ is negligible in the final formula.
\end{proof}

\subsection{Bigraph model}\label{sub:enum-bigraph-model}
Now we compute the probability of having certain neighborhood sizes in a degree-constrained model of bipartite graphs. This time we use \cite{CGM08} to derive the necessary initial probability computation.

\begin{proposition}\label{prop:bigraph-expectation}
There are $\epsilon_{\ref{prop:bigraph-expectation}}, C_{\ref{prop:bigraph-expectation}} > 0$ so the following holds. Let $n\ge C_{\ref{prop:bigraph-expectation}}$ and $(\log n)^{-1/4}\le p\le 1-(\log n)^{-1/4}$. Suppose $(\log n)^{-1/4}\le m/n\le(\log n)^{1/4}$. Let $(\mbf{s},\mbf{t})\in E_{m,n}$ (so $\mbf{s}$ has length $m$) such that each $s_w=pn+O(\sqrt{p(1-p)}n^{1/2+\epsilon_{\ref{prop:bigraph-expectation}}})$ and each $t_w = pm + O(\sqrt{p(1-p)}n^{1/2+\epsilon_{\ref{prop:bigraph-expectation}}}$. Let $G$ be a uniformly random bigraph on vertex sets $W' = \{v_1',\ldots,v_m'\}$ and $W = \{v_1,\ldots,v_n\}$ with these degree sequences between the parts. Consider a size $h$ subset $V\subseteq W$ satisfying $\min(h,n-h)\ge n/(\log n)^{1/8}$, and an integer $t\in[0,s_m]$. For $w\in W$, define $\beta_w$ by $t_w = pm+\beta_w\sqrt{p(1-p)m}$. Let $s_m = pn+\alpha\sqrt{p(1-p)n}$. Then
\begin{align*}
&\mb{P}[\deg_V(v_m') = t]\\
&= (1+O(n^{-1/8}))\exp\bigg(\frac{(\sum_{i=1}^n\beta_i)(\sum_{i=1}^n\beta_i-2\sqrt{mn}\alpha)}{2mn}\bigg)\frac{\binom{h}{t}\binom{n-h}{s_m-t}}{\binom{n}{s_m}}\times\\
&\mb{E}_{\substack{S_1\sim\binom{V}{t}\\S_2\sim\binom{W\setminus V}{s_m-t}\\S=S_1\cup S_2}}\exp\bigg(-\sqrt{\frac{p}{1-p}}\sum_{i\in W}\bigg(-\frac{1-p}{p}\bigg)^{\mbm{1}_S(i)}\frac{\beta_i}{\sqrt{m}}-\frac{1}{2}\sum_{i\in W}\bigg(\frac{1-p}{p}\bigg)^{2\mbm{1}_S(i)-1}\frac{\beta_i^2}{m}\bigg).
\end{align*}
Here $S_1,S_2$ are uniform over their respective domains.
\end{proposition}

As in \cref{sub:enum-graph-model}, there are various corollaries of this fact by computing out what the expectation term yields. The proofs are exactly analogous to the ones given before and consist of routine computation given those ideas. Therefore, we leave out the proofs and merely record the necessary results.

\begin{proposition}\label{prop:bigraph-bounded}
Assume the hypotheses of \cref{prop:graph-expectation}. Assume additionally that $\beta_w = O((\log n)^2)$ for all $w\in W$ and $\alpha = O((\log n)^2)$. Then if $|t-ph|\ge n^{1/2}(\log n)^5$ we have
\[\mb{P}[\deg_V(v_m') = t]\le\exp(-\Omega((t-ph)^2/n)).\]
If $|t-ph|\le n^{3/5}$ and furthermore $(n/m)\cdot(\sum_{i=1}^n\beta_i^2/n)\le(\log n)^{1/9}$ then we have
\begin{align*}
&\mb{P}[\deg_V(v_m') = t]\\
&= (1+O(n^{-1/10}))\frac{\binom{h}{t}\binom{n-h}{s_m-t}}{\binom{n}{s_m}}\exp\bigg[\frac{(\sum_{i=1}^n\beta_i)(\sum_{i=1}^n\beta_i-2\sqrt{mn}\alpha)}{2mn}\\
&-\sqrt{\frac{p}{1-p}}\bigg(\sum_{i\in V}\bigg(1-\frac{t}{ph}\bigg)\frac{\beta_i}{\sqrt{m}}+\sum_{i\in W\setminus V}\bigg(1-\frac{s_m-t}{p(n-h)}\bigg)\frac{\beta_i}{\sqrt{m}}\bigg)-\frac{1}{2}\sum_{i\in W}\frac{\beta_i^2}{m}\\
&+\frac{1}{2mh}\sum_{i < j\in V}(\beta_i-\beta_j)^2 + \frac{1}{2m(n-h)}\sum_{i < j\in V^c}(\beta_i-\beta_j)^2\bigg].
\end{align*}
\end{proposition}
\begin{remark}
Note that in a random bipartite graph with part sizes $m$ and $n$ the condition on $\sum_{i=1}^n\beta_i^2/n$ can only hold when $m$ and $n$ are within a small power of $\log n$ factor.
\end{remark}

\begin{proposition}\label{prop:bigraph-balanced}
Assume the hypotheses of the second part of \cref{prop:bigraph-bounded}. Assume additionally that $|h-n/2| = O(\sqrt{n\log n})$ and $\sum_{i=1}^n\beta_i = O(n^{5/6})$. Then for $|\gamma|\le n^{1/10}$ with $ph+\gamma\sqrt{p(1-p)n}\in\mb{Z}$, we have
\[\mb{P}[\deg_V(v_m') = ph+\gamma\sqrt{p(1-p)n}] = \frac{\sqrt{2}+O(n^{-1/10})}{\sqrt{\pi p(1-p)n}} \exp\bigg(-\frac{1}{2}\bigg(2\gamma-\alpha-\frac{\sum_{i\in V}\beta_i}{\sqrt{mn}/2}\bigg)^2\bigg).\]
\end{proposition}

\subsection{Computational preliminaries}\label{sub:enum-computations}
Now we turn to justifying \cref{prop:graph-expectation,prop:bigraph-expectation}. We first record the graph and bigraph enumeration results that will be used.
\begin{theorem}[{\cite{MW90}}]\label{thm:enum-graph}
There exists a fixed constant $\varepsilon>0$ such that the following holds. Consider a degree sequence $\mbf{d} = (d_1, \dots, d_n)$ of length $n$ such that each $|d_i-\ol{d}|\le n^{1/2+\varepsilon}$, where $\ol{d}=(1/n)\sum_{i=1}^n d_i$ satisfies $\ol{d}\ge n/\log n$. Letting $r = \ol{d}n/2\in\mb{Z}$, $\mu = \ol{d}/(n-1)$, and $\gamma_2^2 = (1/(n-1)^2)\sum_{i=1}^n(d_i-\ol{d})^2$, the number of labelled graphs with degree sequence $\mbf{d}$ is
\[(1+O(n^{-1/4}))\exp\left(\frac{1}{4}-\frac{\gamma_2^2}{4\mu^2(1-\mu)^2}\right)\binom{n(n-1)/2}{r}\binom{n(n-1)}{2r}^{-1}\prod_{i=1}^n\binom{n-1}{d_i}.\]
\end{theorem}
\begin{theorem}[{\cite{CGM08}}]\label{thm:enum-bigraph}
There exists a fixed constant $\varepsilon>0$ such that the following holds. For a pair of integers $n,m \in \mb{N}$ with $n/(\log n)^{1/2}\le m\le n(\log n)^{1/2}$, fix a pair of degree sequences $\mbf{s} = (s_1, \dots, s_n), \mbf{t} = (t_1, \dots, t_m)$ such that each $|s_i-\ol{s}|\le n^{1/2+\varepsilon}$ and $|t_i-\ol{t}| \le m^{1/2+\varepsilon}$, where $\ol{s}=(1/n)\sum_{i=1}^n s_i$ and $\ol{t}=(1/m)\sum_{i=1}^m t_i$ satisfy $\ol{s}\ge n/(\log n)^{1/2}$ and $\ol{t}\ge m/(\log m)^{1/2}$. Let $\gamma_2(s)^2 = (1/n^2)\sum_{i=1}^n(s_i-\ol{s})^2$, $\gamma_2(t)^2 = (1/m^2)\sum_{i=1}^m(t_i-\ol{t})^2$ and $\mu = \sum_{i=1}^n s_i/(mn) = \sum_{i=1}^m t_i/(mn)$. Let $r = \mu mn$. Then the number of labelled bipartite graphs whose partition classes have degree sequences $\mbf{s}$ and $\mbf{t}$ is
\[(1+O(n^{-1/8}))\exp\left(-\frac{1}{2}\left(1-\frac{\gamma_2(s)^2}{\mu(1-\mu)}\right)\left(1-\frac{\gamma_2(t)^2}{\mu(1-\mu)}\right)\right)\binom{mn}{r}^{-1}\prod_{i=1}^n\binom{m}{s_i}\prod_{i=1}^m\binom{n}{t_i}.\]
\end{theorem}

We next compute a certain ratio of binomials that will show up when computing probabilities via graph enumeration. An analogous result for $p = 1/2$ was shown in \cite[Lemma~B.3]{FKNSS21}.
\begin{lemma}\label{lem:binomial-approximation}
Suppose that $n/(\log n)^{1/2}\le m\le n(\log n)^{1/2}$ and $(\log n)^{-1/4}\le p\le 1-(\log n)^{-1/4}$.
\begin{itemize}
    \item If $r = pmn + \Delta_1$ and $d = pm + \Delta_2$ where $\Delta_1 = O(n^{8/5})$ and $\Delta_2 = O(n^{3/5})$, then
    \[\frac{\binom{m(n-1)}{r-d}^{-1}}{\binom{mn}{r}^{-1}}p^d(1-p)^{m-d} = \exp\bigg(\frac{\Delta_1(\Delta_1-2n\Delta_2)}{2p(1-p)mn^2}+O(n^{-1/6})\bigg).\]
    \item If $r = p\binom{n}{2}+\Delta_1$ and $d = p(n-1) + \Delta_2$ where $\Delta_1 = O(n^{8/5})$ and $\Delta_2 = O(n^{3/5})$, then
    \[\frac{\binom{(n-1)(n-2)/2}{r-d}\binom{(n-1)(n-2)}{2r-2d}^{-1}}{\binom{n(n-1)/2}{r}\binom{n(n-1)}{2r}^{-1}}p^d(1-p)^{n-1-d} = \exp\bigg(\frac{2\Delta_1(\Delta_1-n\Delta_2)}{p(1-p)n^3} + O(n^{-1/6})\bigg).\]
\end{itemize}
\end{lemma}
\begin{proof}
For the first expression, we have
\begin{align*}
&\frac{\binom{m(n-1)}{r-d}^{-1}}{\binom{mn}{r}^{-1}}p^d(1-p)^{m-d}\\
&= (1+O(n^{-1/2}))\frac{\Big(\frac{r-d}{m(n-1)}\Big)^{r-d}\Big(1-\frac{r-d}{m(n-1)}\Big)^{m(n-1)-(r-d)}}{\Big(\frac{r}{mn}\Big)^r\Big(1-\frac{r}{mn}\Big)^{mn-r}}p^d(1-p)^{m-d}\\
&=(1+O(n^{-1/2}))\Big(\frac{n}{n-1}\Big)^{m(n-1)}\Big(1-\frac{d}{r}\Big)^{r-d}\Big(1-\frac{m-d}{mn-r}\Big)^{m(n-1)-(r-d)}\Big(\frac{pmn}{r}\Big)^d\Big(\frac{(1-p)mn}{mn-r}\Big)^{m-d}\\
&= \exp\bigg(m - \frac{m}{2(n-1)} -d+\frac{d^2}{r}-\frac{d^2(r-d)}{2r^2} -(m-d)+\frac{(m-d)^2}{mn-r}-\frac{(m-d)^2(mn-r-m+d)}{2(mn-r)^2}\\
&\qquad\qquad+\frac{d(pmn-r)}{r} - \frac{d(r-pmn)^2}{2r^2}+\frac{(m-d)(r-pmn)}{mn-r}-\frac{(m-d)(r-pmn)^2}{2(mn-r)^2}+ O(n^{-1/6})\bigg)\\
&= \exp\bigg(- \frac{m}{2n} +\frac{(pm)^2}{pmn}-\frac{(pm)^2(pmn)}{2(pmn)^2}+\frac{(m-pm)^2}{mn-pmn}-\frac{(m-pm)^2(mn-pmn)}{2(mn-pmn)^2}\\
&\qquad\qquad+\frac{d(pmn-r)}{r} - \frac{d(r-pmn)^2}{2(pmn)^2}+\frac{(m-d)(r-pmn)}{mn-r}-\frac{(m-d)(r-pmn)^2}{2(mn-pmn)^2}+ O(n^{-1/6})\bigg)\\
&= \exp\bigg(\frac{m(r-dn)(r-pmn)}{r(mn-r)} -\frac{1}{2}\Big(\frac{d}{p^2}+\frac{m-d}{(1-p)^2}\Big)\Big(\frac{r-pmn}{mn}\Big)^2+O(n^{-1/6})\bigg)\\
&= \exp\bigg(\frac{m(r-dn)(r-pmn)}{pmn(mn-pmn)} -\frac{m}{2p(1-p)}\Big(\frac{r-pmn}{mn}\Big)^2+O(n^{-1/6})\bigg)\\
&= \exp\bigg(\frac{(r-pmn)(2(r-dn)-(r-pmn))}{2p(1-p)mn^2}+O(n^{-1/6})\bigg)
\end{align*}
We have used Stirling's formula and that $1+x = \exp(x-x^2/2 + O(x^3))$ repeatedly. Now using the definition of $\Delta_1,\Delta_2$ finishes. For the second expression, write
\begin{align*}
\frac{\binom{(n-1)(n-2)/2}{r-d}\binom{(n-1)(n-2)}{2r-2d}^{-1}}{\binom{n(n-1)/2}{r}\binom{n(n-1)}{2r}^{-1}}p^d(1-p)^{n-1-d} = \frac{\frac{\binom{m_1(n_1-1)}{r_1-d_1}^{-1}}{\binom{m_1n_1}{r_1}^{-1}}p^{d_1}(1-p)^{m_1-d_1}}{\frac{\binom{m_2(n_2-1)}{r_2-d_2}^{-1}}{\binom{m_2n_2}{r_2}^{-1}}p^{d_2}(1-p)^{m_2-d_2}}
\end{align*}
where $m_1 = 2m_2 = 2(n-1)$ and $n_1 = n_2 = n/2$, and $r_1 = 2r_2 = 2r$ and $d_1 = 2d_2 = 2d$. Now apply the first part twice.
\end{proof}

\subsection{Proof of \texorpdfstring{\cref{prop:graph-expectation,prop:bigraph-expectation}}{Propositions B.1 and B.6}}\label{sub:enum-proofs}
We first compute the graph version.
\begin{proof}[Proof of \cref{prop:graph-expectation}]
We have that our vertex is the last vertex $v_n$, corresponding to degree $d_n$. Given $S\subseteq W\setminus v_n$ of size $d_n$ (which we abusively identify with a set of integers), let $\mbf{d}_S = (d_1-\mbm{1}_{1\in S},\ldots,d_{n-1}-\mbm{1}_{n-1\in S})$. As in \cref{thm:enum-graph}, let
\begin{align*}
\ol{d} = \frac{1}{n}\sum_{i=1}^nd_i;&\qquad\qquad\qquad\ol{d}_S = \frac{1}{n-1}\sum_{i=1}^{n-1}d_{S,i} = \frac{n}{n-1}\ol{d}-\frac{2d_n}{n-1}\\
r = \frac{\ol{d}n}{2};&\qquad\qquad\qquad r_S = \frac{\ol{d}_S(n-1)}{2} = r-d_n,\\
\mu = \frac{\ol{d}}{n-1};&\qquad\qquad\qquad\mu_S = \frac{\ol{d}_S}{n-2} = \frac{n}{n-2}\mu-\frac{2d_n}{(n-1)(n-2)},\\
\gamma_2^2 = \frac{1}{(n-1)^2}\sum_{i=1}^n(d_i-\ol{d})^2;&\qquad\qquad\qquad\gamma_2^2(S) = \frac{1}{(n-2)^2}\sum_{i=1}^{n-1}(d_{S,i}-\ol{d}_S)^2.
\end{align*}

Note that $\mbf{d}$ and each $\mbf{d}_S$ clearly satisfies the conditions of \cref{thm:enum-graph} due to our given hypotheses. Note that
\[\gamma_2^2(S) = \gamma_2^2+O(n^{-1/4}), \mu_S = \mu+O(1/n)\]
due to the given hypotheses. Now define
\[A = \frac{\binom{(n-1)(n-2)/2}{r-d_n}\binom{(n-1)(n-2)}{2r-2d_n}^{-1}}{\binom{n(n-1)/2}{r}\binom{n(n-1)}{2r}^{-1}}p^{d_n}(1-p)^{n-1-d_n}\]
and recall $d_i = p(n-1)+\beta_i\sqrt{p(1-p)(n-1)}$. We have
\[r-p\binom{n}{2} = \frac{1}{2}\sum_{i=1}^n(d_i-p(n-1)) = \frac{\sqrt{p(1-p)(n-1)}}{2}\sum_{i=1}^n\beta_i.\]
By the given hypotheses and \cref{lem:binomial-approximation} we therefore derive
\[A = \exp\bigg(\frac{(\sum_{i=1}^n\beta_i)(\sum_{i=1}^n\beta_i-2n\beta_n)}{2n^2}+O(n^{-1/6})\bigg)\]

We therefore see from \cref{thm:enum-graph} that
\begin{align*}
&\mb{P}[N(v_n) = S]\\
&= (1+O(n^{-1/4}))\frac{\exp\left(\frac{1}{4}-\frac{\gamma_2^2(S)}{4\mu_S^2(1-\mu_S)^2}\right)\binom{(n-1)(n-2)/2}{r_S}\binom{(n-1)(n-2)}{2r_S}^{-1}\prod_{i=1}^{n-1}\binom{n-2}{d_i-\mbm{1}_S(i)}}{\exp\left(\frac{1}{4}-\frac{\gamma_2^2}{4\mu^2(1-\mu)^2}\right)\binom{n(n-1)/2}{r}\binom{n(n-1)}{2r}^{-1}\prod_{i=1}^n\binom{n-1}{d_i}}\\
&= (1+O(n^{-1/4}))\frac{\binom{(n-1)(n-2)/2}{r_S}\binom{(n-1)(n-2)}{2r_S}^{-1}\prod_{i=1}^{n-1}\binom{n-2}{d_i-\mbm{1}_S(i)}}{\binom{n(n-1)/2}{r}\binom{n(n-1)}{2r}^{-1}\prod_{i=1}^n\binom{n-1}{d_i}}\\
&= (1+O(n^{-1/4}))\frac{A}{\binom{n-1}{d_n}}p^{-d_n}(1-p)^{-(n-1-d_n)}\prod_{i\in S}\frac{d_i}{n-1}\prod_{i\notin S}\frac{n-1-d_i}{n-1}\\
&= (1+O(n^{-1/4}))\frac{A}{\binom{n-1}{d_n}}\prod_{i\in S}\bigg(1+\beta_i\sqrt{\frac{1-p}{p(n-1)}}\bigg)\prod_{i\notin S}\bigg(1-\beta_i\sqrt{\frac{p}{(1-p)(n-1)}}\bigg)\\
&= \frac{A}{\binom{n-1}{d_n}}\exp\bigg(-\sqrt{\frac{p}{1-p}}\sum_{i=1}^{n-1}\bigg(-\frac{1-p}{p}\bigg)^{\mbm{1}_S(i)}\frac{\beta_i}{\sqrt{n-1}}-\frac{1}{2}\sum_{i=1}^{n-1}\bigg(\frac{1-p}{p}\bigg)^{2\mbm{1}_S(i)-1}\frac{\beta_i^2}{n-1}+O(n^{-1/4})\bigg)
\end{align*}
for each $S\in\binom{[n-1]}{d_n}$. Therefore
\begin{align*}
&(1+O(n^{-1/4}))\frac{\binom{n-1}{d_n}}{A\binom{h-\mbm{1}_V(v_n)}{t}\binom{n-h-\mbm{1}_{V^c}(v_n)}{d_n-t}}\mb{P}[\deg_V(v_n) = t]\\
&= \mb{E}_{\substack{S_1\sim\binom{V\setminus v_n}{t}\\S_2\sim\binom{V^c\setminus v_n}{d_n-t}\\S=S_1\cup S_2}}\exp\bigg(-\sqrt{\frac{p}{1-p}}\sum_{i=1}^{n-1}\bigg(-\frac{1-p}{p}\bigg)^{\mbm{1}_S(i)}\frac{\beta_i}{\sqrt{n-1}}-\frac{1}{2}\sum_{i=1}^{n-1}\bigg(\frac{1-p}{p}\bigg)^{2\mbm{1}_S(i)-1}\frac{\beta_i^2}{n-1}\bigg),
\end{align*}
where we are taking the uniform distribution for the sets $S_1,S_2$ over their domains. Rearranging gives the desired result.
\end{proof}

Now we compute the bipartite version.
\begin{proof}[Proof of \cref{prop:bigraph-expectation}]
We have that our vertex is the last vertex $v_m'$, corresponding to degree $s_m$. Given $S\subseteq W$ of size $s_m$ (which we abusively identify with a set of integers), let $\mbf{t}_S = (t_1-\mbm{1}_{1\in S},\ldots,t_n-\mbm{1}_{n\in S})$. As in \cref{thm:enum-bigraph}, let
\begin{align*}
\ol{t} = \frac{1}{n}\sum_{i=1}^nt_i;&\qquad\qquad\qquad\ol{t}_S = \frac{1}{n}\sum_{i=1}^nt_{S,i} = \ol{t}-\frac{s_m}{n}\\
r = \ol{t}n;&\qquad\qquad\qquad r_S = \ol{t}_Sn = r-s_m,\\
\mu = \frac{\ol{t}}{m};&\qquad\qquad\qquad\mu_S = \frac{\ol{t}_S}{m-1} = \frac{m}{m-1}\mu-\frac{s_m}{(m-1)n},\\
\gamma_2^2(t) = \frac{1}{n^2}\sum_{i=1}^n(t_i-\ol{t})^2;&\qquad\qquad\qquad\gamma_2^2(t_S) = \frac{1}{n^2}\sum_{i=1}^n(t_{S,i}-\ol{t}_S)^2.
\end{align*}
Let $\mbf{s}'$ be $\mbf{s}$ restricted to the first $m-1$ values, and let
\[\ol{s}' = \frac{1}{m-1}\sum_{i=1}^{m-1}s_i,\qquad\gamma_2^2(s') = \frac{1}{(m-1)^2}\sum_{i=1}^{m-1}(s_i-\ol{s}')^2.\]

Note that $(\mbf{s},\mbf{t})$ and each $(\mbf{s}',\mbf{t}_S)$ clearly satisfy the conditions of \cref{thm:enum-bigraph} due to our given hypotheses. Note that
\[\gamma_2^2(t) = \gamma_2^2(t_S)+O(n^{-1/4}),\qquad\mu_S = \mu+O(n^{-3/4}),\qquad\gamma_2^2(s') = \gamma_2^2(s) + O(n^{-1/4})\]
due to the given hypotheses.

Now define
\[A = \frac{\binom{(m-1)n}{r-s_m}^{-1}}{\binom{mn}{r}^{-1}}p^{s_m}(1-p)^{n-s_m}\]
and recall $t_i = pm+\beta_i\sqrt{p(1-p)m}$ and $s_m = pn + \alpha\sqrt{p(1-p)n}$. We have
\[r-pmn = \sum_{i=1}^n(t_i-pm) = \sqrt{p(1-p)m}\sum_{i=1}^n\beta_i.\]
By the given hypotheses and \cref{lem:binomial-approximation} (with $m,n$ switched) we therefore derive
\[A = \exp\bigg(\frac{(\sum_{i=1}^n\beta_i)(\sum_{i=1}^n\beta_i-2\sqrt{mn}\alpha)}{2mn}+O(n^{-1/6})\bigg)\]

We therefore see from \cref{thm:enum-bigraph} (with $m$ and $n$ switched) that
\begin{align*}
&\mb{P}[N(v_m') = S]\\
&= (1+O(n^{-1/8}))\frac{\exp\left(-\frac{1}{2}\left(1-\frac{\gamma_2(s')^2}{\mu(1-\mu)}\right)\left(1-\frac{\gamma_2(t_S)^2}{\mu(1-\mu)}\right)\right)\binom{(m-1)n}{r}^{-1}\prod_{i=1}^{m-1}\binom{n}{s_i}\prod_{i=1}^n\binom{m-1}{t_i-\mbm{1}_S(i)}}{\exp\left(-\frac{1}{2}\left(1-\frac{\gamma_2(s)^2}{\mu(1-\mu)}\right)\left(1-\frac{\gamma_2(t)^2}{\mu(1-\mu)}\right)\right)\binom{mn}{r}^{-1}\prod_{i=1}^m\binom{n}{s_i}\prod_{i=1}^n\binom{m}{t_i}}\\
&= (1+O(n^{-1/8}))\frac{\binom{(m-1)n}{r}^{-1}\prod_{i=1}^{m-1}\binom{n}{s_i}\prod_{i=1}^n\binom{m-1}{t_i-\mbm{1}_S(i)}}{\binom{mn}{r}^{-1}\prod_{i=1}^m\binom{n}{s_i}\prod_{i=1}^n\binom{m}{t_i}}\\
&= (1+O(n^{-1/8}))\frac{A}{\binom{n}{s_m}}p^{-s_m}(1-p)^{-(n-s_m)}\prod_{i\in S}\frac{t_i}{m}\prod_{i\notin S}\frac{m-t_i}{m}\\
&= (1+O(n^{-1/8}))\frac{A}{\binom{n}{s_m}}\prod_{i\in S}\bigg(1+\beta_i\sqrt{\frac{1-p}{pm}}\bigg)\prod_{i\notin S}\bigg(1-\beta_i\sqrt{\frac{p}{(1-p)m}}\bigg)\\
&= \frac{A}{\binom{n}{s_m}}\exp\bigg(-\sqrt{\frac{p}{1-p}}\sum_{i=1}^n\bigg(-\frac{1-p}{p}\bigg)^{\mbm{1}_S(i)}\frac{\beta_i}{\sqrt{m}}-\frac{1}{2}\sum_{i=1}^n\bigg(\frac{1-p}{p}\bigg)^{2\mbm{1}_S(i)-1}\frac{\beta_i^2}{m}+O(n^{-1/8})\bigg)
\end{align*}
for each $S\in\binom{[n]}{s_m}$. Therefore
\begin{align*}
&(1+O(n^{-1/8}))\frac{\binom{n}{s_m}}{A\binom{h}{t}\binom{n-h}{s_m-t}}\mb{P}[\deg_V(v_m') = t]\\
&= \mb{E}_{\substack{S_1\sim\binom{V}{t}\\S_2\sim\binom{W\setminus V}{s_m-t}\\S=S_1\cup S_2}}\exp\bigg(-\sqrt{\frac{p}{1-p}}\sum_{i=1}^n\bigg(-\frac{1-p}{p}\bigg)^{\mbm{1}_S(i)}\frac{\beta_i}{\sqrt{m}}-\frac{1}{2}\sum_{i=1}^n\bigg(\frac{1-p}{p}\bigg)^{2\mbm{1}_S(i)-1}\frac{\beta_i^2}{m}\bigg),
\end{align*}
where we are taking the uniform distribution for the sets $S_1,S_2$ over their domains. Rearranging gives the desired result.
\end{proof}

\end{document}